\documentclass[11pt,reqno,oneside]{amsart}

\usepackage{graphicx}
\usepackage{amssymb}
\usepackage{epstopdf}
\usepackage{subfigure}
\usepackage[justification=centering]{caption}
\usepackage[a4paper, total={6in, 9.1in}]{geometry}

\usepackage{amsmath,amsfonts,amsthm,mathrsfs,amssymb,cite}

\usepackage[usenames]{color}
\usepackage{ulem}

\usepackage{graphics}
\usepackage{framed}

\usepackage{enumerate}
\usepackage{hyperref}
\usepackage[capitalize,noabbrev]{cleveref}

\numberwithin{equation}{section}

\newtheorem{thm}{Theorem}[section]

\newtheorem{lem}{Lemma}[section]
\newtheorem{prop}{Proposition}[section]
\theoremstyle{definition}
\newtheorem{defn}{Definition}[section]

\newtheorem{rem}{Remark}[section]
\numberwithin{equation}{section}

\def \Vh0{\stackrel{\circ}{V}_h}

\def\brho{{\boldsymbol \rho }}

\newcommand{\lc}
{\mathrel{\raise2pt\hbox{${\mathop<\limits_{\raise1pt\hbox
{\mbox{$\sim$}}}}$}}}

\newcommand{\gc}
{\mathrel{\raise2pt\hbox{${\mathop>\limits_{\raise1pt\hbox{\mbox{$\sim$}}}}$}}}

\newcommand{\ec}
{\mathrel{\raise2pt\hbox{${\mathop=\limits_{\raise1pt\hbox{\mbox{$\sim$}}}}$}}}

\def\bb{\begin{equation}} \def\ee{\end{equation}}

\def\beqn{\begin{eqnarray}}  \def\eqn{\end{eqnarray}}

\def\beqnx{\begin{eqnarray*}} \def\eqnx{\end{eqnarray*}}

\def\bn{\begin{enumerate}} \def\en{\end{enumerate}}

\def\bd{\begin{description}} \def\ed{\end{description}}



\newcommand{\bfx}{\mathbf{x}}

\def\bsi{{\mathrm{i}}}
\def\Oh{{\mathcal  O}}
\def\bsi{{\mathrm  i}}
\def \rmd {{\mathrm  d}}

 \allowdisplaybreaks \allowdisplaybreaks[4]

\title[Spectral properties of the acoustoelastic transmission eigenvalue problem]{On a coupled-physics transmission eigenvalue problem and its spectral properties with applications }

\author{Huaian Diao}
\address{School of Mathematics and Key Laboratory of Symbolic Computation and Knowledge Engineering of Ministry of Education, Jilin University,
Changchun, Jilin 130012, China.}
\email{diao@jlu.edu.cn}

\author{Hongyu Liu}
\address{Department of Mathematics, City University of Hong Kong, Kowloon, Hong Kong SAR, China.}
\email{hongyu.liuip@gmail.com; hongyliu@cityu.edu.hk}

\author{Qingle Meng}
\address{Department of Mathematics, City University of Hong Kong, Kowloon, Hong Kong SAR, China.}
\email{mengq12021@foxmail.com; qinmeng@cityu.edu.hk}

\author{Li Wang}
\address{Department of Mathematics, City University of Hong Kong, Kowloon, Hong Kong SAR, China.}
\email{liwangmath12@126.com; lwang637-c@my.cityu.edu.hk}
\begin{document}
\maketitle

\begin{abstract}
In this paper, we investigate a transmission eigenvalue problem that couples the principles of acoustics and elasticity. This problem naturally arises when studying fluid-solid interactions and constructing bubbly-elastic structures to create metamaterials. We uncover intriguing local geometric structures of the transmission eigenfunctions near the corners of the domains, under typical regularity conditions. As applications, we present novel unique identifiability and visibility results for an inverse problem associated with an acoustoelastic system, which hold practical significance.

\medskip

\noindent{\bf Keywords:}~~acoustoelastic;  transmission eigenvalue problem; corner singularity; spectral geometry; inverse inclusion problem; unique identifiability

\noindent{\bf 2010 Mathematics Subject Classification:}~~58J50; 35P25; 74F10; 35A02; 74B05.

\end{abstract}

\section{Introduction}

\subsection{Mathematical formulation}

Initially focusing on mathematics, but not physics, we present the formulation of the coupled-physics transmission eigenvalue problem for our study. Let $\Omega\Subset\mathbb{R}^N$, $N=2, 3$, be a bounded Lipschitz domain with a connected complement $\mathbb{R}^N\backslash\overline{\Omega}$. Let $\rho_e, \rho_b$ and $\kappa$ be positive real-valued $L^\infty(\Omega)$ functions that are bounded below, and $\lambda, \mu$ be $L^\infty(\Omega)$ functions satisfying
\begin{equation}\label{eq:cond1}
2\lambda+N\mu>0\quad\mbox{in}\ \ \Omega.
\end{equation}
Let $\gamma \in\mathbb{C}$ and consider the following eigenvalue problem for $\mathbf{u}\in H^1(\Omega)^N$ and $v\in H^1(\Omega)$:
\begin{equation}\label{eq:trans1}
\begin{cases}
\frac{1}{\rho_e}\mathcal{L}_{\lambda, \mu} \mathbf{u}+\gamma \mathbf{u}={\bf 0}\quad & \mbox{in}\ \Omega,\medskip\\
\kappa\nabla\cdot(\rho_b^{-1}\nabla v)+\gamma v=0\quad & \mbox{in}\ \Omega,\medskip\\
\mathbf{u}\cdot\nu-\frac{1}{\rho_b\gamma} \nabla v\cdot\nu=0\quad & \mbox{on}\ \Gamma,\medskip\\
T_{\nu}\mathbf{u}+v\nu={\bf 0}\quad & \mbox{on}\ \Gamma,
\end{cases}
\end{equation}
where $\nu\in\mathbb{S}^{N-1}:=\{\mathbf{x}\in\mathbb{R}^N; |\mathbf{x}|=1\}$ represents the exterior normal direction to $\partial \Omega$, $\Gamma\Subset\partial\Omega$,
\begin{equation}\label{eq:lame}
\mathcal{L}_{\lambda, \mu} \mathbf{u}:=\mu\Delta \mathbf{u}+(\lambda+\mu)\nabla\nabla\cdot \mathbf{u},
\end{equation}
and
\begin{equation}\label{eq:traction1}
T_\nu\mathbf{u}=\lambda(\nabla\cdot\mathbf{u})\nu+2\mu(\nabla^s\mathbf{u})\nu\quad\mbox{with}\ \ \nabla^s\mathbf{u}:=\frac 1 2(\nabla\mathbf{u}+\nabla\mathbf{u}^t).
\end{equation}
Here the superscript $``t"$ signifies the matrix transpose. It is clear that $\mathbf{u}\equiv\mathbf{0}$ and $v\equiv 0$ are a pair of trivial solutions to \eqref{eq:trans1}. If there exist nontrivial solutions to \eqref{eq:trans1}, $\gamma\in\mathbb{C}$ is called a transmission eigenvalue, and the corresponding nontrivial solutions $( v, \mathbf{u})$ are called the transmission eigenfunctions. If $\Gamma=\partial\Omega$, then \eqref{eq:trans1} is referred to as the full-data transmission eigenvalue problem, otherwise it is referred to as the partial-data transmission eigenvalue problem.

The first and second equations in \eqref{eq:trans1} are eigenvalue problems associated with elliptic partial differential operators (of different physical natures). However, one can verify (though not straightforwardly) that the transmission eigenvalue problem \eqref{eq:trans1} is non-elliptic and non-selfadjoint. In what follows, we shall be mainly concerned with real transmission eigenvalues $\gamma\in\mathbb{R}_+$, which are physically more related. Set $\omega:=\sqrt{\gamma}\in\mathbb{R}_+$. The transmission eigenvalue problem \eqref{eq:trans1} can then be rewritten as 
\begin{equation}\label{eq:trans2}
\begin{cases}
\mathcal{L}_{\lambda, \mu} \mathbf{u}+\omega^2\rho_e\mathbf{u}={\bf 0}\quad & \mbox{in}\ \Omega,\medskip\\
\nabla\cdot(\rho_b^{-1}\nabla v)+\omega^2\kappa^{-1} v=0\quad & \mbox{in}\ \Omega,\medskip\\
\mathbf{u}\cdot\nu-\frac{1}{\rho_b\omega^2} \nabla v\cdot\nu=0\quad & \mbox{on}\ \Gamma,\medskip\\
T_{\nu}\mathbf{u}+v\nu={\bf 0}\quad & \mbox{on}\ \Gamma.
\end{cases}
\end{equation}
In what follows, we also refer to $\omega$ as a transmission eigenvalue.

\subsection{Motivations and background}\label{sec:MB}

The transmission eigenvalue problem \eqref{eq:trans1} arises from the related studies of two coupled-physics processes that are of practical significance. Note that the first equation in \eqref{eq:trans1} corresponds to the Lam\'{e} system, which governs the deformation field $\mathbf{u}$ in an elastic medium. The parameters $\lambda$, $\mu$, and $\rho_e$ represent the bulk moduli and density of the elastic medium, respectively.
On the other hand, the second equation in \eqref{eq:trans1} corresponds to the Helmholtz equation, which governs the propagation of the acoustic wave field $v$ in an acoustic medium. The parameters $\rho_b$ and $\kappa$ denote the density and modulus of the acoustic medium, respectively.
Furthermore, the angular frequency of the two physical waves is denoted by $\omega\in\mathbb{R}_+$.

To motivate the study of the transmission eigenvalue problem \eqref{eq:trans1}, let us consider the first physical scenario where an acoustic wave is scattered by an elastic body immersed in a fluid. This problem is of central importance in detecting and identifying submerged objects and is generally referred to as the fluid-solid or wave-structure interaction problem \cite{Monk}. Let $\rho_b$ and $\kappa$ be positive constants that characterize a uniformly homogeneous background space, namely the acoustical property of the fluid. Let $v^i$ be an entire solution to
\begin{equation*}
\nabla\cdot(\rho_b^{-1}\nabla v^i)+\omega^2\kappa^{-1} v^i=0\quad \mbox{in}\ \ \mathbb{R}^N,
\end{equation*}
which signifies an impinging acoustic wave field. Let $(\Omega; \lambda, \mu, \rho_e)$ be an elastic inclusion as described above. The scattering of the impinging acoustic field $v^i$ by the elastic body $(\Omega; \lambda, \mu, \rho_e)$ is described by
\begin{equation}\label{eq:transae1}
\begin{cases}
\mathcal{L}_{\lambda, \mu} \mathbf{u}+\omega^2\rho_e\mathbf{u}={\bf 0}\quad & \mbox{in}\ \Omega,\medskip\\
\nabla\cdot(\rho_b^{-1}\nabla v)+\omega^2\kappa^{-1} v=0\quad & \mbox{in}\ \mathbb{R}^N\backslash\overline\Omega,\medskip\\
\mathbf{u}\cdot\nu-\frac{1}{\rho_b\omega^2} \nabla v\cdot\nu=0\quad & \mbox{on}\ \partial\Omega,\medskip\\
T_{\nu}\mathbf{u}+ v\nu={\bf 0}\quad & \mbox{on}\ \partial\Omega,\medskip\\
v=v^i+ v^s\quad &\mbox{in}\ \mathbb{R}^N\backslash\overline\Omega,\medskip\\
\lim\limits_{r\rightarrow \infty}r^{(N-1)/2}(\partial_r-\mathrm{i}k)v^s=0,\ & r:=|\mathbf{x}|,
\end{cases}
\end{equation}
where $k:=\omega\sqrt{\rho_b/\kappa}$ is known as the acoustic wave number, $\mathrm{i}=\sqrt{-1}$ and $v, v^s$ are respectively referred to as the total and scattered acoustic wave fields. The last limit is known as the Sommerfeld radiation condition, which holds uniformly in the angular variable $\hat{\mathbf{x}}:=\mathbf{x}/r\in\mathbb{S}^{N-1}$. The scattered field $v^s$ admits the following asymptotic expansion:
\begin{equation}\label{eq:f1}
v^s(\mathbf{x})=\frac{e^{\mathrm{i}k r}}{r^{(N-1)/2}} v_\infty({\hat{\mathbf{x}}})+\mathcal{O}(r^{-(N+1)/2})
\end{equation}
as $r\rightarrow +\infty$. Furthermore, $v_\infty$ is called the acoustic far-field pattern encoding the information of the elastic inclusion $(\Omega; \lambda, \mu, \rho_e)$ from the measurement of the induced acoustic wave field. It is worth emphasizing that the correspondence between $v_\infty$ and $v^s$ is one-to-one; see \cite{CK, Monk}.

Let us introduce an operator $\mathcal{F}$ that maps the elastic body to its corresponding acoustic far-field pattern via Equation \eqref{eq:transae1} as follows:
\begin{equation}\label{eq:operator1}
\mathcal{F}(\hat{\mathbf{x}}; (\Omega; \lambda, \mu, \rho_e), v)=v_\infty(\hat{\mathbf{x}}),\ \ \hat{\mathbf{x}}\in\mathbb{S}^{N-1}.
\end{equation}
It is direct to verify that $\mathcal{F}$ is a nonlinear and completely continuous operator due to the well-posedness of the system described by Equation \eqref{eq:transae1}. The problem of detecting a submerged object mentioned earlier can be recast as an inverse problem where the goal is to invert the operator $\mathcal{F}$. However, this inverse problem is inherently ill-posed due to the analyticity of $v_\infty$. In approaching the inverse problem \eqref{eq:operator1}, we naturally aim to characterize the range of the operator $\mathcal{F}$. However, we will take a completely different perspective by considering the kernel space of $\mathcal{F}$, which is defined as follows:
\begin{equation}\label{eq:kernel1}
\mathcal{F}(\hat{\mathbf{x}}; (\Omega; \lambda, \mu, \rho_e), v)=0.
\end{equation}
It turns out that the geometric characterization of the kernel space of $\mathcal{F}$ can be used to establish novel results for the inverse problem \eqref{eq:operator1} in certain challenging scenarios.
 On the other hand, it is noted that if \eqref{eq:kernel1} occurs, one has $v_\infty\equiv 0$, and hence $v^s\equiv 0$ in $\mathbb{R}^N\backslash\overline\Omega$ by the one-to-one correspondence between two of them. In such a case, the elastic inclusion $(\Omega; \lambda, \mu, \rho_e)$ is neutral for the acoustoelastic effect due to the acoustic impinging field $v^i$. From a probing point of view, this also means that the elastic body is invisible for the exterior measurement. Therefore, it is of great interest and significance to investigate the existence of such neutral inclusions and understand the quantitative behavior of the wave fields inside the elastic body when invisibility occurs. This exploration not only satisfies our curiosity from a physical perspective but also sheds light on the fundamental understanding of the system. By straightforward analysis, one can show that if $v_\infty\equiv 0$, then $( v|_{\Omega}, \mathbf{u})$ satisfies \eqref{eq:trans2} with $\Gamma=\partial\Omega$, namely, $\omega^2$ is an acoustoelastic transmission eigenvalue and $( v|_{\Omega}, \mathbf{u})$ is the corresponding pair of transmission eigenfunctions. Hence, the study of the transmission eigenvalue problem \eqref{eq:trans1} provides a broader spectral perspective to understand those both theoretically and practically important issues discussed above. It was proved in \cite{DLLT} that under a generic scenario, there exist acoustoelastic transmission eigenvalues to \eqref{eq:trans1}, which form a discrete set and only accumulate to infinity \cite{KR,Monk}.

The other physical scenario concerns the construction of metamaterials by bubbly elastic structures (cf.\cite{Ammari,Ammari1,Ammari2,LiLiuZou}). Here, one employs special structures with air bubbles embedded in elastic mediums to realize peculiar elastic properties in a homogenized sense. The  physical process is described by
\begin{equation}\label{eq:transea1}
\begin{cases}
\nabla\cdot(\rho_b^{-1}\nabla v)+\omega^2\kappa^{-1} v=0\quad & \mbox{in}\ \Omega,\medskip\\
\mathcal{L}_{\lambda, \mu} \mathbf{u}+\omega^2\rho_e\mathbf{u}=\mathbf{0}\quad & \mbox{in}\ \mathbb{R}^N\backslash\overline{\Omega},\medskip\\
\mathbf{u}\cdot\nu-\frac{1}{\rho_b\omega^2} \nabla v\cdot\nu=0\quad & \mbox{on}\ \partial\Omega,\medskip\\
T_{\nu}\mathbf{u}+v\nu=\mathbf{0}\quad & \mbox{on}\ \partial\Omega,\medskip\\
\mathbf{u}=\mathbf{u}^{i}+\mathbf{u}^{s}\quad &\mbox{in}\ \mathbb{R}^N\backslash\overline\Omega,\medskip\\
\mbox{$\mathbf{u}^s$ satisfies the Kupradze radiation condition}.
\end{cases}
\end{equation}
In this context, we assume that $\lambda, \mu$, and $\rho_e$ are constant, representing the homogeneous configuration of the elastic background medium. The region $(\Omega; \rho_b, \kappa)$ corresponds to the acoustic inclusion. Equation \eqref{eq:transea1} represents a complete solution to
\begin{equation}\label{eq:ei1}
\mathcal{L}_{\lambda, \mu}\mathbf{u}^{i}+\omega^2\rho_e\mathbf{u}^{i}=0\quad\mbox{in}\ \ \mathbb{R}^N,
\end{equation}
which characterizes the incident elastic field. On the other hand, $\mathbf{u}^{s}$ denotes the scattered field generated by the acoustic inclusion. The Kupradze radiation condition, given by
\begin{equation}\label{eq:kp1}
\lim_{r\rightarrow\infty}r^{\frac{n-1}{2}}\left(\frac{\partial \mathbf{u}_\beta^{s}}{\partial r}-\mathrm{i}k_\beta \mathbf{u}_\beta^{s}\right) =\,\mathbf 0,\quad \beta=p, s,
\end{equation}
is imposed. Here, $\mathbf{u}^{s}$ is decomposed into its pressure component $\mathbf{u}_{p}^{s }$ and shear component $\mathbf{u}_{s}^{s}$, defined as
\begin{equation}\label{eq:decomp1}
\mathbf{u}^{s}=\mathbf{u}_{p}^{s }+\mathbf{u}_{s}^{s},\ \	 \mathbf{u}_{p}^{s}:=-\frac{1}{k_{p}^{2}} \nabla\left( \nabla \cdot \mathbf{ u}^{s}\right ), \quad \mathbf{ u}_{s}^{s}:=\begin{cases}  \frac{1}{k_{s}^{2}}\nabla \times \nabla \times  {\mathbf u}^{s}\ \ & (\mbox{3D})\\
 \frac{1}{k_{s}^{2}} \bf{curl} \operatorname{curl} u^{\mathrm {s} }\ \ & (\mbox{2D})	
\end{cases},
\end{equation}
where
\begin{equation}\label{eq:kpks}
	k_p :=\frac{\omega\sqrt{\rho_e}}{\sqrt{ 2\mu+\lambda }} \quad \mbox{and} \ \ k_s:=\frac{\omega\sqrt{\rho_e}}{\sqrt{ \mu}}.
\end{equation}
In \eqref{eq:decomp1}, the two-dimensional operators $\bf{curl}$  and $\operatorname{curl}$ are defined respectively by
\[
 {\rm curl}\, \mathbf{ H}=\partial_1 H_2-\partial_2 H_1, \quad {\bf
curl}\, {h}=(\partial_2 h, -\partial_1 h)^\top,
\]
with $\mathbf{H}=(H_1, H_2)$ and $H$ being vector-valued and scalar functions, respectively. Similar to \eqref{eq:f1}, one has the following asymptotic expansions as $r=|\mathbf{x}|\rightarrow\infty$:
\begin{equation}\label{eq:far-field}
\begin{split}
\mathbf{u}^{{s}}(\mathbf{x})= \frac{{e}^{\mathrm{i} k_{p} r}}{{r}^{\frac{
N-1}{2} }} u_{p}^{\infty}(\hat{\mathbf{x} } ) \hat{\mathbf{x}}+\frac{{e}^{\mathrm{i} k_{s} r}}{{r}^{\frac{
N-1}{2} }} u_{s}^{\infty}(\hat{\mathbf{x} }) \hat{\mathbf{x} }^{\perp}+\mathcal{O}\left(\frac{1}{r^{(N+1) / 2}}\right) ,
\end{split}
\end{equation}
where $$	\mathbf{u}_{\beta}^{{s}}(\mathbf{x}) = \frac{{e}^{\mathrm{i} k_{\beta} r}}{{r}^{\frac{
	N-1}{2} }}\left\{u_{\beta}^{\infty}(\hat{\mathbf{x}}) \hat{\mathbf{x}}+\mathcal{O}\left(\frac{1}{r}\right)\right\},\ \ \beta=p, s.$$
Define $\mathbf{u}_t^{\infty}(\hat{\mathbf{x}} ) :=u_{p}^{\infty}(\hat{\mathbf{x}}) \hat{\mathbf{x}}+u_{s}^{\infty}(\hat{\mathbf{x}}) \hat{\mathbf{x}}^{\perp}$, which is referred to as the far-field pattern of $\mathbf{u}^{s}$. Furthermore, similar to \eqref{eq:operator1}, one can introduce the following abstract formulation of an inverse problem:
\begin{equation}\label{eq:so1}
\mathcal{S}\big((\Omega; \rho_b, \kappa), \mathbf{u}^{i}\big)=\mathbf{u}_t^\infty(\hat{\mathbf{x}}), \quad \hat{\mathbf{x}}\in\mathbb{S}^{N-1},
\end{equation}
where $\mathcal{S}$ is implicitly defined via the system \eqref{eq:transea1}. In a completely similar manner to \eqref{eq:kernel1}, one can show that the study of the kernel space of $\mathcal{S}$ naturally leads to the study of the transmission eigenvalue problem \eqref{eq:trans2} associated with $(v,\mathbf{u}^{i}|_{\Omega})$. That is, if $\mathbf{u}_t^\infty\equiv \mathbf{0}$, or equivalently $\mathbf{u}^{s}=0$ in $\mathbb{R}^N\backslash\overline{\Omega}$, one would have that $(v,\mathbf{u}^{i}|_{\Omega})$ is a pair of transmission eigenfunctions to \eqref{eq:trans2}. 


\subsection{Discussion on existing results and summary of main findings}

The acoustoelastic (AE) transmission eigenvalue problem \eqref{eq:trans2} has been considered in the literature in different contexts. For a particular case with $v\equiv 0$ and $\Gamma=\partial\Omega$, it is known as the Jones eigenvalue problem \cite{Monk, Luke}. Excluding the Jones eigenvalues guarantees the well-posedness of the PDE system \eqref{eq:transae1}.  There is a rich literature about the Jones eigenvalues; see \cite{Monk, Barucq} and the references cited therein.  Since the transmission eigenvalue problem \eqref{eq:trans2} has its complicated physical and mathematical nature that arises from the coupling of two different physical processes, there is not much progress on \eqref{eq:trans2}. In fact,  the transmission eigenvalue problem \eqref{eq:trans2} was investigated in \cite{KR,Monk} in the context of reconstruction schemes for the inverse problem \eqref{eq:operator1}.  It was proved that if the transmission eigenvalues exist for \eqref{eq:trans2}, they must form a discrete set and accumulate only at infinity.  Recently,   in \cite{DLLT}  it was demonstrated the existence of the transmission eigenvalue for \eqref{eq:trans2} under certain generic physical conditions. Furthermore,  a geometric rigidity result for the transmission eigenfunctions was established, which showed that there exist transmission functions that tend to localize on the boundary of the underlying domain.  For further developments on transmission eigenvalue problems in inverse scattering, see \cite{CCH23, CH13}. 



As discussed in the previous subsection, the AE transmission eigenvalue problem \eqref{eq:trans2} is closely related to the kernel space  of $\mathcal F$ and $\mathcal S$ defined by \eqref{eq:kernel1} and \eqref{eq:so1} respectively. When the kernel spaces, associated with the scattering problem \eqref{eq:transae1} and \eqref{eq:transea1} respectively, are not empty, the corresponding elastic inclusion $(\Omega; \lambda, \mu, \rho_e)$ and  acoustic inclusion  $(\Omega; \rho_b, \kappa)$, may not be detected through external observations, leading to their invisibility or transparency.  However, geometric singularities in the inclusion's shape, such as corners or points of high curvature, hinder invisibility and transparency, as demonstrated in \cite{BPS2014,CX21,PSV2017,EH18,BL2021} for time-harmonic acoustic scattering. New methods from free boundary problems, introduced in \cite{CV23,PS21}, have been applied to study the visibility or invisibility of isotropic acoustic inclusions with Lipschitz or less regular boundaries. The anisotropic acoustic case was further investigated in \cite{CVX23,KSS24}.




On the other hand, \cite{BL2017} explored the local geometric structure of acoustic transmission eigenfunctions near corners within the underlying domain, demonstrating that these eigenfunctions vanish near corners under certain regularity conditions. Further developments in elastic and electromagnetic transmission eigenvalue problems can be found in \cite{BLin19, DLS, BLx2021}. An important application of understanding the geometric structure of transmission eigenfunctions near corners is in establishing the unique shape identifiability of the inclusion using a single far-field measurement. This measurement refers to collecting far-field data in all observation directions for a fixed incident frequency and wave, as opposed to many far-field measurements. The problem of determining shape from a single far-field measurement, known in the literature as Schiffer's problem, has a rich history in inverse scattering problems (see \cite{COL, CK} and the references therein). To date, uniqueness results from a single far-field measurement hold only when certain a priori information about the inclusion's shape and size is known. For corresponding developments on unique identifiability in (conductive) acoustic, elastic, and electromagnetic medium scattering problems using a single far-field measurement, see \cite{BLin19, DCL, DLS, BLx2021,HSV16}.

In this paper,  we conduct a  comprehensive study of the  geometric property of AE transmission eigenfunctions $(v, \mathbf{u})$ to \eqref{eq:trans2} around a corner on the boundary of the underlying domain $\Omega$, under certain regularity assumptions.  Indeed, the geometric properties of AE transmission eigenfunctions $(v, \mathbf{u})$ can be summarized as follows:
\begin{itemize}
\item[a)] In $\mathbb R^2$, when $v$ is H\"older continuous and $\mathbf u$ has $C^{1,\alpha_1}$ ($0<\alpha_1<1$) regularity in a neighborhood of the corner $\mathbf x_c\in \partial \Omega$, we show that the strain  $\nabla^s \mathbf u(\mathbf x_c)$ of $\mathbf u=(u_1,u_2)^\top$ at $\mathbf x_c$ defined by \eqref{eq:traction1} must be a scalar matrix, depending solely on $\partial_{x_1} u_1(\mathbf x_c)$. For more details, see \eqref{eq:thm21}. 
\item[b)] In $\mathbb R^2$, under the assumptions in a), an intriguing algebraic characterization of $v(\mathbf x_c)$ and $\partial_{x_1} u_1(\mathbf x_c)$ with respect to the Lam\'e constants $\lambda$ and $\mu$ of the homogenous elastic background medium is obtained. For more details, see \eqref{eq:thm21 v0}. 
\item[c)] In $\mathbb R^2$, under the same setup in b), when $\partial_{x_1} u_1(\mathbf x_c)=0$, we reveal the vanishing property of $v$ at $\mathbf x_c$. Furthermore, if $\partial_{x_2} u_1(\mathbf x_c)=0$, we can prove that $\nabla \mathbf u(\mathbf x_c)$ is a zero matrix (see Theorem \ref{thm:thm21} for more details). 
\item[d)] The findings in a)-c) also hold at the edge corner point  in $\mathbb R^3$ under certain assumptions on AE transmission eigenfunctions $v$ and $\mathbf u$; see  Theorem \ref{thm:3Ru123} for more detailed discussions. 
\end{itemize}
We emphasize that these regularity assumptions on $(v, \mathbf{u})$  can be satisfied when investigating  the inverse problem \eqref{eq:so1} associated with the scattering problem \eqref{eq:transea1}.  These  geometric characterizations of  AE transmission eigenfunctions $(v, \mathbf{u})$ will be applied to examine the visibility of an acoustic inclusion containing a corner embedded in the elastic homogeneous background medium; see Theorem \ref{thm:radiating} for further discussions. Furthermore, by considering the scattering problem \eqref{eq:transea1}, we demonstrate the unique recovery of a convex polygonal acoustic inclusion using at most three  far-field measurements, leveraging the novel geometric properties of AE transmission eigenfunctions $(v, \mathbf{u})$ around the corner.  The unique results are detailed in Theorem \ref{thm:inverse1}. 


To establish the geometric structures of the AE transmission eigenfunctions as mentioned, it is essential to analyze the singularities of the underlying transmission eigenfunctions induced by the geometric irregularities of the domain, particularly at the corners. For this purpose, we develop a microlocal analysis localized around the corner. A crucial component of our study is an integral identity that involves the difference between the AE transmission eigenfunctions and a specific type of Complex Geometric Optics (CGO) solutions. Unlike most existing research, our work incorporates boundary integral terms arising from the coupled transmission conditions. The intricate nature of this coupled system introduces significant challenges in both analysis and estimation. By carefully analyzing the asymptotic behavior of the CGO solutions, we are able to achieve the desired results.

The remaining sections of the paper are structured as follows. In Section \ref{sec:2}, we derive the geometrical singularities of the transmission eigenfunctions near a corner in the two dimensional case. In Section \ref{sec:3}, we extend this analysis to the three dimensional case. Section \ref{sec:ip} focuses on the study of uniqueness in determining a polygonal acoustic inclusion in \eqref{eq:transea1} using at most three far-field measurements. Furthermore,  we demonstrate that an acoustic medium inclusion with a convex planar corner must scatter for any incident compressional or shear elastic wave. 




\section{Local geometrical properties near corners of acoustic-elastic transmission eigenfunctions: two-dimensional case}\label{sec:2}

This section is devoted to studying the geometric properties of the transmission eigenfunction pair $( v, \mathbf{u})$ about \eqref{eq:trans2} near the planar corner. Next, we introduce some notations for the geometric setup; see Fig.~\ref{fig:3} for a schematic illustration. In $\mathbb{R}^2$, let $(r, \theta)$ represent the polar coordinates of $\mathbf{x}$, where $\mathbf{x} = (x_1, x_2)^\top \in \mathbb{R}^2$. Thus, $r = |\mathbf{x}|$ and $\theta = \mathrm{arg}(x_1 + \bsi x_2)$. Furthermore, let $B_h$ denote the central disk at the origin with a radius of $h \in \mathbb{R}_+$. We consider the strictly convex sector defined by
\begin{equation}\label{eq:W}
W = \left\{ \mathbf{x} \in \mathbb{R}^2; \theta_m \leqslant \mathrm{arg}(x_1 + \bsi x_2) \leqslant \theta_M \right\},
\end{equation}
which is formed by the two half-lines $\Gamma^-$ and $\Gamma^+$. Here, $-\pi < \theta_m < \theta_M < \pi$ such that $0 < \theta_M - \theta_m < \pi$, and $\theta_m$ and $\theta_M$ correspond to the polar angles of $\Gamma^-$ and $\Gamma^+$, respectively.

 Considering the invariance of operators $\mathcal{L}_{\lambda, \mu}$ and $\Delta$ under rigid motions, we assmue that the Lipschitz domain $\Omega \Subset \mathbb{R}^2$ in \eqref{eq:trans2} contains a corner $\Omega \cap B_h = \Omega \cap W$ with the vertex located at $\mathbf{0} \in \Gamma$. Here, $W$ represents the sector as defined in \eqref{eq:W}, and $h \in \mathbb{R}+$ denotes a suitably small neighborhood. Specifically, we define the neighborhood as
\begin{equation}\label{eq:Sh}
S_h = \Omega \cap B_h = \Omega \cap W.
\end{equation}
Throughout the following discussion, we use the notations
\begin{equation}\label{eq:gammaEX}
\Gamma_h^{\pm} = \Gamma^{\pm} \cap B_h,\quad \Lambda_h = W \cap \partial B_h,
\end{equation}
and assume that $\Gamma$ in \eqref{eq:trans2} satisfies the condition that $\Gamma_h^{\pm} \Subset \Gamma$.

\begin{figure}[h]
\centering
\subfigure{\includegraphics[width=0.30\textwidth]{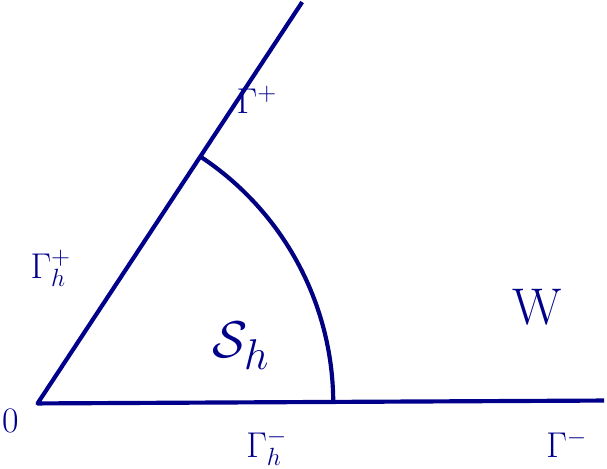}}
\caption{Schematic illustration of a 2D corner. \label{fig:3}}
\end{figure}

For the sake of simplicity, we assume that $\rho_b$ and $\kappa$ are constants in \eqref{eq:trans2}. For any transmission eigenfunction $(v, \mathbf{u}) \in H^1(\Omega) \times H^1(\Omega)^2$ to  \eqref{eq:trans3}, it is direct to imply that  $(v, \mathbf{u})$ satisfies the PDE system: 
\begin{equation}\label{eq:trans3}
\begin{cases}
\mathcal{L}_{\lambda, \mu} \mathbf{u}+\omega^2\rho_e\mathbf{u}=\mathbf{0}\quad & \mbox{in}\ S_h,\medskip\\
\Delta v+\omega^2 \rho_b\kappa^{-1} v=0\quad & \mbox{in}\ S_h,\medskip\\
\mathbf{u}\cdot\nu-\frac{1}{\rho_b\omega^2} \nabla v\cdot\nu=0\quad & \mbox{on}\ \Gamma_h^\pm,\medskip\\
T_{\nu}\mathbf{u}+v\nu=\mathbf{0}\quad & \mbox{on}\ \Gamma_h^\pm.
\end{cases}
\end{equation}

\begin{rem}
 Throughout this paper, the parameters $\lambda, \mu, \rho_e, \rho_b,$ and $\kappa$ are all real-valued. Without loss of generality, we can assume that the transmission eigenfunctions $(v, \mathbf{u})$ to  \eqref{eq:trans3} is also real-valued, and $\omega$ belongs to the set of positive real numbers $\mathbb R_+$. Alternatively, we can deal with the real and imaginary parts of $(v, \mathbf{u})$, respectively. 
\end{rem}

The following lemma can be obtained by using the Laplace transform and integral estimation.

\begin{lem}
\label{lem:int r }
Let $a \in \mathbb{C}$ and $\Re(a)>0$. For any given positive numbers $\alpha$ and  $h$ satisfying $0<h <  e $, if $    \Re(a) \geq \frac{2 \alpha}{ e}$, one has
 \begin{equation}\label{eq:1r1}
  \int_{0}^{h} r^{\alpha} {e}^{  -a r} \rmd r = \frac{  \Gamma ( \alpha + 1)}{a ^ {\alpha + 1}} +  \Oh({e}^{ -\frac{h \Re(a)}{2}})
 \end{equation}
 as $\Re(a)\rightarrow +\infty. $
 \end{lem}

 The CGO solutions, commonly used to investigate the geometric properties of $( v, \mathbf{u})$ at planar corners, are constructed as follows. In Lemma \ref{lem:1}, we present the CGO solution related to the Laplacian operator  $\Delta$. The CGO solution associated with the Lam\'e operator $\mathcal L_{\lambda,\mu}$ is detailed in Lemma \ref{lem:u0}.

\begin{lem}\label{lem:1}
For the strictly convex sector $W$ defined in \eqref{eq:W}, there exists a vector
\begin{align}\label{eq:d def}
	 \mathbf d :=(\cos \phi, \sin \phi )^\top \in \mathbb S^1, \quad (\theta_M +\pi/2< \phi< \theta_m +3\pi/2)
\end{align}
such that
\begin{equation}\label{eq:d cond}
	\mathbf d  \cdot \boldsymbol{\tau} \leq -\delta<0, \quad \forall\, \boldsymbol{\tau} \in W \cap \mathbb S^1,
\end{equation}
where $\delta>0$ is a constant depending  on $W$ and $\mathbf d$. Let $\mathbf d^\perp$ be the unit vector obtained via rotating $\mathbf d$ anti-clockwise by $\pi/2$, namely, 
\begin{equation}\label{eq:d}
\mathbf d^\perp:=(-\sin \phi, \cos \phi )^\top.
\end{equation}
It is obvious that  $\mathbf d \perp \mathbf d^\perp$. Denote
	\begin{equation}\label{eq:cgo}
	v_0(\mathbf  x):= e^{ \brho  \cdot \mathbf x }, \quad \brho=s(\mathbf d+\mathrm i \mathbf d^\perp)
	,\quad  s\in \mathbb R_+ .
	\end{equation}
	Then
\begin{equation}\label{eq:v0 delta}
		\Delta v_0(\mathbf x)=0\quad \mbox{in}\quad R^2 .
\end{equation}
Furthermore, for $\alpha ,\,s>0$, it holds that
		{
        \begin{subequations}
		\begin{align}
            & \int_{\Gamma_h^\pm}|v_0(\mathbf{ x})| |\mathbf{x} |^\alpha {\rm d}\sigma \leq \frac{ \Gamma(\alpha+1) }{ (\delta s)^{\alpha+1}} + \Oh(e^{-s\delta h/2} ), \label{eq:v0 gamma al}\\
			& \int_{0}^h r^\alpha e^{r \boldsymbol{\rho} \cdot \boldsymbol \tau}  {\rm d}{ r} =\frac{  \Gamma(\alpha+1) }{ ( -\boldsymbol{\rho} \cdot \boldsymbol \tau )^{\alpha+1}} + \Oh(e^{-s\delta h/2} ),
			\quad \forall \,\boldsymbol \tau\in W\cap \mathbb S^1, \label{eq:v0 gamma}\\
			& \int_{S_h}|v_0(\mathbf{ x})| |\mathbf{x} |^\alpha {\rm d}\mathbf{ x} \leq \frac{(\theta_M-\theta_m) \Gamma(\alpha+2) }{ (\delta s)^{\alpha+2}} + \Oh(e^{-s\delta h/2} ), \label{eq:v0 alpha}\\
& \int_{S_h} v_0( \bfx ) {\rm d} \bfx = \int_{S_h} e^{ \brho \cdot \mathbf x } {\rm d} \bfx =  \frac{\bsi e^{2\bsi
			\phi }}{ 2 s^{2}}\left(e^{-2\bsi \theta_M} - e^{-2\bsi \theta_m}\right)+\Oh(e^{-s\delta h/2} ),\label{eq:v0sh} 
			\end{align}
		\end{subequations}
	as $s\rightarrow +\infty$.
	}
\end{lem}


\begin{proof}
According to $\brho \cdot \brho=0$, one obtains \eqref{eq:v0 delta}.  Let $\mathbf x=r \boldsymbol{\tau} $, where $\boldsymbol{\tau}=(\cos \theta, \sin \theta)^\top$is the polar coordinate of $\mathbf x$. Then, based on \eqref{eq:d cond}, it is evident that
\begin{subequations}
\begin{align}
		& \brho \cdot \mathbf x=sr e^{\bsi (\theta-\phi)} 
	, \quad \forall \mbox{ }\mathbf x \in W \label{eq:rho x},\\
& \brho \cdot \mathbf x = s(\mathbf d+\mathrm i \mathbf d^\perp) \cdot r \boldsymbol{\tau} = s r \mathbf d \cdot \boldsymbol{\tau} + \bsi s r \mathbf d^\perp \cdot \boldsymbol{\tau} \mbox{ , }
|e^{\brho \cdot \mathbf x}| = |e^{s r \mathbf d \cdot \boldsymbol{\tau}} |\leq e^{-s r\delta } \label{eq:rho x2}.
    \end{align}
    \end{subequations}
    
	Since $|\mathbf x|=|r \boldsymbol{\tau} |= r$, by applying Lemma \ref{lem:int r } and combining it with \eqref{eq:rho x2}, we obtain
\begin{equation}\nonumber
\int_{S_h}|v_0(\mathbf{ x})| |\mathbf{x} |^\alpha {\rm d}\mathbf{ x} \leq \int_{0}^h \int_{\theta_m}^{\theta_M} e^{-s r\delta } r^{\alpha + 1} \rmd \theta \rmd r = \frac{(\theta_M-\theta_m) \Gamma(\alpha+2) }{ (\delta s)^{\alpha+2}} + \Oh(e^{-s\delta h/2} ).
\end{equation}

 Using polar coordinates and Lemma  \ref{lem:int r }, and considering  \eqref{eq:d cond}, \eqref{eq:rho x} and \eqref{eq:rho x2}, it is direct to calculate
	\begin{equation}\nonumber
		\begin{split}
			\int_{S_h} v_0(\mathbf x )\rmd \mathbf  x = & \int_{0}^h \int_{\theta_m}^{\theta_M}r e^{ \brho \cdot \mathbf x }\rmd \theta \rmd r = \int_{0}^h \int_{\theta_m}^{\theta_M} r e^{sr e^{\bsi (\theta-\phi)}} \rmd \theta \rmd r \\
=& \frac{\bsi e^{2\bsi
			\phi }}{ 2 s^{2}}\left(e^{-2\bsi \theta_M} - e^{-2\bsi \theta_m}\right)
			+\Oh(e^{-s\delta h/2} )\quad \mbox{as}\quad s\rightarrow +\infty.
		\end{split}
	\end{equation}
	Similarly, we can obtain \eqref{eq:v0 gamma al} and \eqref{eq:v0 gamma} by using Lemma \ref{lem:int r } and \eqref{eq:rho x2}.
	
	The proof is complete. 
\end{proof}

\begin{lem}\label{lem:u0}
	Under the same setup of Lemma \ref{lem:1}, let
	\begin{align}\label{eq:p def}
			\mathbf p=  \mathbf d^\perp-\bsi \mathbf d,
	\end{align}
 where $\mathbf d \in \mathbb S^1$ and $\mathbf d^\perp\in \mathbb S^1 $ are defined in \eqref{eq:d def} and \eqref{eq:d}, respectively. Denote
	\begin{equation}\label{eq:u0 cgo}
		\mathbf u_0(\mathbf x)=\mathbf p\, v_0(\mathbf x).
	\end{equation}
	Then
	\begin{equation}\label{eq:u0 lame}
		\mathcal{L}_{\lambda, \mu} \mathbf u_0(\mathbf x)=\mathbf 0 \quad \mbox{in}\quad R^2,
	\end{equation}
 where $\mathcal{L}_{\lambda, \mu}$ is defined in \eqref{eq:lame}. For any positive parameters $\alpha, s$, it holds that
\begin{equation}\label{eq:u0 alpha}
		\int_{S_h}|\mathbf u_0(\mathbf{ x})| |\mathbf{x} |^\alpha {\rm d}\mathbf{ x} \leq \frac{\sqrt{2}(\theta_M-\theta_m) \Gamma(\alpha+2) }{ (\delta s)^{\alpha+2}} + \Oh(e^{-s\delta h/2} ),
		\end{equation}
	and 
		\begin{equation}\label{eq:u0sh}
		\int_{S_h} \mathbf u_0( \bfx ) {\rm d} \bfx =  \frac{\bsi e^{2\bsi
				\phi }}{ 2 s^{2}}\left(e^{-2\bsi \theta_M} - e^{-2\bsi \theta_m}\right) \mathbf p+\Oh(e^{-s\delta h/2} ),
	\end{equation}
	
	as $s\rightarrow +\infty$.

\end{lem}

\begin{proof}
	Since $\mathbf p \cdot \boldsymbol \rho=0$, 
	 where $ \boldsymbol \rho$ is defined in \eqref{eq:cgo}, by directly calculations, one has
	\begin{equation}\label{eq:u0 div}
		\nabla \cdot \mathbf u_0(\mathbf x)=0.
	\end{equation}
	Substituting \eqref{eq:v0 delta} and \eqref{eq:u0 div} into \eqref{eq:lame}, it is immediately to obtain
	 \eqref{eq:u0 lame}.
Since
\begin{equation}\label{eq:Shuv}
  \int_{S_h}|\mathbf u_0(\mathbf{ x})| |\mathbf{x} |^\alpha {\rm d}\mathbf{ x}=\sqrt{2}\int_{S_h}|\mathbf v_0(\mathbf{ x})| |\mathbf{x} |^\alpha {\rm d}\mathbf{ x},
\end{equation}
 by virtue of \eqref{eq:v0 alpha} and \eqref{eq:Shuv}, we can easily get \eqref{eq:u0 alpha}. Then combining \eqref{eq:v0sh} with \eqref{eq:u0 cgo}, we successfully attain \eqref{eq:u0sh}.
	\end{proof}

The following two propositions provide properties for the vectors $\boldsymbol{\rho}$ and $\mathbf{p}$ as defined in \eqref{eq:d def} and \eqref{eq:p def}, respectively. These properties will be utilized in our subsequent analysis. 

	\begin{prop}
		Let $\mathbf d$ and $\mathbf p$ be defined by \eqref{eq:d def} and \eqref{eq:p def} respectively.  	One has
	\begin{align}
	& \boldsymbol{\rho}=s  e^{-\mathrm i \phi} \mathbf{e}_1\mbox{, }\mbox{ } \mathbf p=- \mathrm{i} e^{-\mathrm i \phi} \mathbf{e}_1, \label{eq:p eqn} 
\end{align}
where $\boldsymbol{\rho} $ is given by \eqref{eq:cgo} and $\mathbf{e}_1 = (1, \bsi)^\top$.
Furthermore, it holds that
\begin{equation}\label{eq:248 rho p}
		\frac{1}{\mathrm i s} \boldsymbol{\rho}=\mathbf p.
	\end{equation}
	\end{prop}
	
	\begin{proof}
		By \eqref{eq:cgo}, we can obtain the formulation of $\boldsymbol{\rho} $. It is ready to know that
\begin{align}
\boldsymbol{\rho} & = s (\mathbf d + \mathrm i \mathbf d^\perp) =\begin{bmatrix}
		\cos \phi \\ \sin \phi
	\end{bmatrix}  + \mathrm i \begin{bmatrix}
		-\sin \phi \\ \cos \phi
	\end{bmatrix}= s e^{-\mathrm i \phi} \mathbf{e}_1,\nonumber  \\
\mathbf p & =\mathbf d^\perp-\mathrm i \mathbf d =\begin{bmatrix}
	-\sin \phi \\ \cos \phi
\end{bmatrix} -\mathrm i \begin{bmatrix}
	\cos \phi \\ \sin \phi
\end{bmatrix}=- \mathrm{i} e^{-\mathrm i \phi} \mathbf{e}_1.\nonumber  
\end{align}
Therefore, it can be directly to deduce \eqref{eq:248 rho p}.
	\end{proof}

	\begin{prop}
		Let $\Gamma_h^\pm$ be defined by \eqref{eq:gammaEX}. Denote
		\begin{equation}\label{eq:taumM}
			\boldsymbol \tau_M=(\cos \theta_M, \sin \theta_M)^\top, \quad \boldsymbol \tau_m=(\cos \theta_m, \sin \theta_m)^\top
		\end{equation}
		by the  unit tangential vectors  to $\Gamma_h^\pm$ respectively. The exterior unit normal vectors to $\Gamma_h^\pm$ are given by
		\begin{equation}\label{eq:nu pm}
			\nu_M=(-\sin \theta_M, \cos \theta_M)^\top, \quad \nu_m=(\sin \theta_m, -\cos \theta_m)^\top.
		\end{equation}
		 Then one has
		 \begin{subequations}
		\begin{alignat}{2}
		  \brho \cdot \nu_M &  =\bsi s e^{\bsi (\theta_M-\phi )}, &\hspace{20pt} \brho \cdot \nu_m & =-\bsi s e^{\bsi (\theta_m-\phi )},\label{eq:rho 220a} \\ 
          \brho \cdot \boldsymbol \tau_M & = s e^{\bsi (\theta_M-\phi )}, &\hspace{20pt} \brho \cdot \boldsymbol\tau_m & = s e^{\bsi (\theta_m-\phi )},\label{eq:rho 220b}\\	 
            \mathbf p \cdot \nu_M & =e^{\bsi (\theta_M-\phi )},  
		 	&\hspace{20pt} \mathbf p \cdot \nu_m & =-e^{\bsi (\theta_m-\phi )}. \label{eq:rho 220c}
		\end{alignat}
		 \end{subequations}
		 where $\phi$ is the polar angle of $\mathbf d$ given in \eqref{eq:d cond},  $\boldsymbol \rho$ and $\mathbf p$ are defined in \eqref{eq:cgo} and \eqref{eq:u0 cgo} respectively.
	\end{prop}
	\begin{proof}
		Using the properties
		\begin{subequations}\nonumber
					\begin{alignat}{2}
			\mathbf d \cdot \boldsymbol \tau_M  &=\cos(\theta_M-\phi ),&\hspace{20pt} \mathbf d \cdot \boldsymbol \tau_m & =\cos(\theta_m-\phi ),\\
			\mathbf d^\perp  \cdot \boldsymbol \tau_M&=\sin(\theta_M-\phi ),&\hspace{20pt} \mathbf d^\perp \cdot \boldsymbol \tau_m  &=\sin(\theta_m-\phi ).
		\end{alignat}
		\end{subequations}
		After some calculations, we derive \eqref{eq:rho 220b}. By noting
		\begin{subequations}\nonumber
					\begin{alignat}{2}
			\mathbf d \cdot \nu_M & =\sin(\phi -\theta_M), &\hspace{20pt} \mathbf d \cdot \nu_m & =-\sin(\phi-\theta_m ),\\
			\mathbf d^\perp  \cdot \nu_M&=\cos(\phi-\theta_M ), &\hspace{20pt} \mathbf d^\perp \cdot \nu_m & =-\cos(\phi-\theta_m ),
		\end{alignat}
		\end{subequations}
it  arrives that 
 \begin{align}\nonumber
\brho \cdot \nu_M & = s(\mathbf d+\mathrm i \mathbf d^\perp)  \cdot \nu_M = s(  \mathbf d \cdot \nu_M + \bsi \mathbf d^\perp \cdot \nu_M),\\
\brho \cdot \nu_m & = s(\mathbf d+\mathrm i \mathbf d^\perp)  \cdot \nu_m = s(  \mathbf d \cdot \nu_m + \bsi \mathbf d^\perp \cdot \nu_m).\nonumber
\end{align}
Hence, we can obtain \eqref{eq:rho 220a}. Similarly, the process of deriving \eqref{eq:rho 220c} is straightforward.
	\end{proof}
	
	The following lemma yields the explicit expression of the traction operator $T_\nu \mathbf u_0(\mathbf x)$, where $\mathbf u_0$ is given by \eqref{eq:u0 cgo}. In Lemma \ref{lem:24 u tu}, we state the estimates for  $L^2$ and $H^1$ norm of $\mathbf u_0$ and $T_\nu \mathbf u_0(\mathbf x)$ on $\Lambda_h$.

	\begin{lem}
		Let $\nu$ be the exterior normal vector to $\Gamma$ and $T_\nu$ be defined in \eqref{eq:traction1}, where $\Gamma \Subset \partial \Omega$. Recall that the CGO solution  $\mathbf u_0$ is given by \eqref{eq:u0 cgo}. Then
		\begin{equation}\label{eq:Tu}
			T_\nu \mathbf u_0(\mathbf x)=\mu [ ( \brho \cdot \nu )\mathbf p+(\mathbf p  \cdot \nu) \brho ] v_0(\mathbf x),
		\end{equation}
		where $v_0(\mathbf x)$ is defined in \eqref{eq:cgo}.
			\end{lem}
 
	\begin{proof}
		Due to \eqref{eq:u0 div}, by the definition \eqref{eq:traction1}, as well as on the basis of the definition of $v_0(\mathbf  x)$ and $\mathbf{u}_0(\mathbf x)$ in \eqref{eq:cgo} and \eqref{eq:u0 cgo}, one yields that 
\begin{align}
T_\nu\mathbf{u}_0(\mathbf x)& =\lambda(\nabla\cdot\mathbf{u}_0(\mathbf x))\nu+2\mu(\nabla^s\mathbf{u}_0(\mathbf x))\nu \label{eq:Tu0} \\
&= 2\mu(\nabla^s\mathbf{u}_0(\mathbf x))\nu= \mu (\nabla\mathbf{u}_0(\mathbf x)+\nabla\mathbf{u}_0(\mathbf x)^t) \nu, \nonumber
\end{align}
 Thereby, \eqref{eq:Tu0} can be rewritten as
 \begin{equation}\label{eq:Tu01}
 T_\nu\mathbf{u}_0(\mathbf x)= \mu [\nabla(\mathbf p \,e^{ \brho  \cdot \mathbf x }) + \nabla(\mathbf p\, e^{ \brho  \cdot \mathbf x })^t]\nu.
 \end{equation}
 By performing direct calculations, it yields that \eqref{eq:Tu}.
	\end{proof}
	
	\begin{lem}\label{lem:24 u tu}
	Suppose that $\Lambda_h$, $v_0(\mathbf x)$, and $\mathbf u_0(\mathbf x)$  be defined by \eqref{eq:gammaEX}, \eqref{eq:cgo} and  \eqref{eq:u0 cgo} respectively. Recall that $\delta>0$ is given in \eqref{eq:d cond}. Then we have
	\begin{subequations}
		\begin{align}
			\|v_0( \mathbf x)\|_{H^1(\Lambda_h )^{2}}& \leq \sqrt {1+8 s^2} \sqrt {h(\theta_M-\theta_m )}\,e^{-s  h \delta} , \label{eq:v H1}\\
			\left	\| \nabla v_0( \mathbf x ) \right \|_{L^2(\Lambda_h )^{2}} &\leq 2 s  \sqrt{ 2 h (\theta_M-\theta_m)}\, e^{-s  h \delta },\label{eq:p v0} \\
			\| \mathbf{u}_0(\mathbf x) \|_{H^1(\Lambda_h )^{2} }
			&\leq 2\sqrt{1+ 32 s^2} \sqrt{h(\theta_M-\theta_m )} \,e^{-s  h \delta},\label{eq:u0 l2}\\
			\| T_{\nu}\mathbf{u}_0(\mathbf x) \|_{L^2(\Lambda_h )^{2} } &\leq 8 \mu (1+s^2 ) \sqrt{h(\theta_M-\theta_m )}\, e^{-s  h \delta}, \label{eq:Tu L2}
		\end{align}
	\end{subequations}
	all of which decay exponentially as $s\rightarrow +\infty. $
\end{lem}

\begin{proof}
Recalling \eqref{eq:d cond}, \eqref{eq:cgo} and combining \eqref{eq:rho x2}, one has
\begin{equation}\label{eq:21 u1norm}
	\| v_0( \mathbf x)\|^2_{L^2(\Lambda_h )^2 }=  r \int_{\theta_m}^{\theta_M} |e^{ 2\boldsymbol \rho  \cdot \mathbf x}|\rmd \theta \leq h \int_{\theta_m}^{\theta_M} \,e^{- 2 s  h \delta} \rmd \theta \leq h\,e^{- 2 s  h \delta}(\theta_M-\theta_m).
\end{equation}
Since
\begin{equation}\label{eq:part v0}
	\frac{\partial v_0}{\partial x_1}=\rho_{1} e^{\boldsymbol \rho  \cdot \mathbf x}, \quad \frac{\partial v_0}{\partial x_2}=\rho_{2} e^{\boldsymbol \rho  \cdot \mathbf x},\quad |\boldsymbol{ \rho}| \leq 2 s,
\end{equation}
where $\brho_1=(\rho_{1},\rho_{2})^\top$,
by virtue of \eqref{eq:d cond} and direct calculations, one has
\begin{equation}\label{eq:22 norm}
\left	\| \nabla v_0( \mathbf x ) \right \|^2_{L^2(\Lambda_h )^{2}} \leq 8 s^2 h\, e^{-2s  h \delta }(\theta_M-\theta_m).
\end{equation}
Thus, it's easy to attain \eqref{eq:p v0}. Combing \eqref{eq:21 u1norm} with \eqref{eq:22 norm},  we can further get the \eqref{eq:v H1}.

Similarly, by the knowledge of  \eqref{eq:u0 cgo} and \eqref{eq:21 u1norm},  we note that
\begin{equation}\label{eq:p ineq}
	|\mathbf p| \leqslant 2.
\end{equation}
Hence, we have
\begin{equation}\label{eq:226 u0l2}
	\| \mathbf u_0( \mathbf x)\|^2 _{L^2(\Lambda_h )^2 } \leqslant 8  h\, e^{- 2 s  h \delta} (\theta_M-\theta_m).
\end{equation}
Since \begin{equation}\notag
	\nabla \mathbf u_0( \mathbf x)=\nabla(\mathbf p  v_0( \mathbf x ))=\begin{bmatrix}
		p_{1} \rho_{1} & p_{1} \rho_{2} \\
		p_{2} \rho_{1} &p_{2} \rho_{2}
		\end{bmatrix}e^{\boldsymbol \rho \cdot \mathbf x}
	=\begin{bmatrix}
		 p_{1} \\ p_{2}
	 \end{bmatrix}
 \begin{bmatrix}
 	\rho_{1} & \rho_{2}
 \end{bmatrix}e^{\boldsymbol \rho \cdot \mathbf x},
\end{equation} 
where $\mathbf p=(p_{1},p_{2})^\top $.
Furthermore, one can imply that
 \begin{equation}\label{gradient u0}
	\left	\| \nabla\mathbf u_0( \mathbf x ) \right \|^2_{L^2(\Lambda_h )^{2}} \leqslant 64 s^2  h e^{- 2 s  h \delta} (\theta_M-\theta_m). 
\end{equation}
Combining \eqref{eq:226 u0l2} with \eqref{gradient u0}, we can obtain \eqref{eq:u0 l2}. 
Clearly, we know that $|\rho_{1}| \leqslant 2s$, $|\rho_{2}| \leqslant 2s$, here $\rho_{1} $ and $\rho_{2}$ are the two components of $\boldsymbol \rho$. Similarly, $|p_{1}| \leqslant 2$ and  $|p_{2}| \leqslant 2$, where $p_{1}$ and $p_{2}$ are the two components of $\mathbf p$. By virtue of \eqref{eq:Tu}, it's simply to obtain the following inequality
\begin{equation}\label{eq:227 tu0}
	\begin{split}
		\| T_{\nu}(\mathbf{u}_0) \|_{L^2(\Lambda_h ) ^{2}}^2&=\mu^2 \left (\int_{\Lambda_h} \left| [(\brho   \cdot \nu ), (\mathbf p  \cdot \nu) ] \begin{bmatrix}
			p_{1}\\ \rho_{1}
		\end{bmatrix}\right|^2 |v_0|^2 \rmd \sigma \right. \\
		&\left.\quad +\int_{\Lambda_h} \left| [(\brho  \cdot \nu ), (\mathbf p  \cdot \nu) ] \begin{bmatrix}
			p_{2}\\ \rho_{2}
		\end{bmatrix}\right|^2 |v_0|^2 \rmd \sigma \right)\\
		&\leq 64 \mu^2 (1+s^2)^2 h e^{-2s  h \delta} (\theta_M-\theta_m ).
	\end{split}
\end{equation}
Therefore, we can easily get \eqref{eq:Tu L2}.
\end{proof}

Next, we recall a special type of Green formula for $H^1$ functions, which shall be needed in establishing a key integral identity for deriving the vanishing property of the transmission eigenfunction.

\begin{lem}\label{lem:green}
	Let $\Omega\Subset \mathbb R^2 $ be a bounded Lipschitz domain. For any $f,g\in H^1_\Delta:=\{f\in H^1(\Omega)|\Delta f\in L^2(\Omega)\}$, there holds the following second Green identity:
	\begin{equation}\label{eq:GIN1}
		\int_\Omega(g\Delta f-f\Delta g)\,\mathrm{d}x=\int_{\partial\Omega}(g\partial_{\nu}f-f\partial_{\nu}g)\,\mathrm{d}\sigma.
	\end{equation}
	Similarly, suppose that $\mathbf{u_1} \in H^1(\Omega )^2 $ and $\mathbf v_1 \in H^1(\Omega )^2$ satisfying $\mathbf{u_1}, \mathbf v_1\in H^1_{ \mathcal L_{\lambda,\mu}}:=\{\mathbf f\in H^1(\Omega)^2|\mathcal L_{\lambda, \mu  } \mathbf f\in L^2(\Omega)^2\}$. The the following Green identity holds
\begin{equation}\label{eq:green1}
	\int_{\Omega}(\mathbf{u}_1 \cdot \mathcal{L}_{\lambda,\mu}\mathbf{v}_{1} - \mathbf{u}_1 \cdot \mathcal{L}_{\lambda,\mu}\mathbf{v}_{1}  ){\rm d}\mathbf{x}=\int_{\partial\Omega}( T_{\mathbf{\nu}}\mathbf{v}_{1} \cdot \mathbf{u}_1-T_{\mathbf{\nu}}\mathbf{u}_1 \cdot \mathbf{v}_{1}) {\rm d}\sigma.
\end{equation}

\end{lem}
\begin{rem}
Lemma \ref{lem:green} is a specific instance of more general results found in \cite[Lemma 3.4]{costabel88} and \cite[Theorem 4.4]{McLean}. Notably, Lemma \ref{lem:green} does not require the
  $H^2$-regularity which is typically needed for the standard Green formula. In particular, for the transmission eigenfunctions $(v,\mathbf u)\in H^1(\Omega)\times H^1(\Omega)^2$ to \eqref{eq:trans3} associated with \eqref{eq:trans3}, we have $v\in H_\Delta^1$ and $\mathbf u\in H^1_{\mathcal L_{\lambda,\mu} }$, thus the Green identities \eqref{eq:GIN1} and \eqref{eq:green1} hold for $v, \mathbf u$. This fact shall be frequently utilized in our subsequent analysis.
\end{rem}

Several integral identities and estimations shall be derived in Lemma \ref{lem:I} to Lemma \ref{lem:estimate}, which will be used to prove Theorem \ref{thm:thm21}.

\begin{lem}\label{lem:I}
Let $\Lambda_h$ be defined in \eqref{eq:gammaEX}. Recall that the CGO solutions $v_0$ and $\mathbf u_0$ are given by \eqref{eq:cgo} and \eqref{eq:u0 cgo}, respectively.
	Denote
	\begin{equation}\label{eq:I lambdah}
		I_{\Lambda_h,1}=\int_{\Lambda_h} (v \partial_\nu v_0-v_0\partial_\nu v)\mathrm d \sigma, \quad I_{\Lambda_h,2}=\int_{\Lambda_h} (\mathbf u \cdot T_\nu \mathbf u_0-\mathbf u_0\cdot T_\nu \mathbf u)\mathrm d \sigma,
	\end{equation}
	where $v \in H^1(\Omega )$ and $u\in H^1(\Omega)^2$ are the acoustoelastic transmission eigenfunctions to \eqref{eq:trans3}. The following integral estimations hold

\begin{subequations}
	\begin{align}
		|I_{\Lambda_h,1}|&\leq \left(\sqrt{1+ 8 s^2}+2 \sqrt{2}s\right)\sqrt{h(\theta_M-\theta_m )} e^{-s h \delta}  \| {v} \|_{H^1(S_h  ) },\label{eq:I arc1} \\
		|I_{\Lambda_h,2 }|&\leq 2 \left[\sqrt{1+ 32 s^2}+ 4 \mu (1+s^2 ) \right]\sqrt{h(\theta_M-\theta_m )} e^{-s h \delta} \| \mathbf{u} \|_{H^1(S_h  ) ^{2} } , \label{eq:I arc2}
	\end{align}
\end{subequations}
	both of which decay exponentially as $s\rightarrow+\infty $. \end{lem}
\begin{proof}

	By using the H{\"o}lder inequality and the trace theorem, one has
\begin{align}
|I_{\Lambda_h,2 }| &\leq  \|(T_{\nu}\mathbf u\|_{H^{-1/2} ( \Lambda_{h} ) ^{2}} \|\mathbf{u}_0 \|_{ H^{1/2}(\Lambda_{h}) ^{2}}+ \| T_{\nu} \mathbf{u}_0 \|_{L^2(\Lambda_h ) ^{2}} \| \mathbf{u} \|_{L^2(\Lambda_h ) ^{2}} \notag \\
&\leq \left(  \|\mathbf{u}_0 \|_{ H^{1}(\Lambda_{h}) ^{2}}  +  \| T_{\nu}\mathbf{u}_0 \|_{L^2(\Lambda_h ) ^{2}}\right) \| \mathbf{u} \|_{H^1(S_h  ) ^{2} } .\label{eq:230 bound}
\end{align}
Substituting \eqref{eq:u0 l2} and \eqref{eq:Tu L2} into \eqref{eq:230 bound}, it's clearly to obtain \eqref{eq:I arc2}. Furthermore, one can deduce \eqref{eq:I arc1} in a similar way.
\end{proof}

To investigate the geometrical singular behavior of transmission eigenfunctions to  \eqref{eq:trans3} near a corner, we need to utilize the CGO solutions to derive an integral identity \eqref{eq:int1 imp} in the following lemma. 

\begin{lem}\label{lem:28 int}
	Let $( v, \mathbf u)\in  H^1(\Omega) \times H^1(\Omega)^2 $ be the transmission eigenfunction to \eqref{eq:trans2} satisfying \eqref{eq:trans3} and under necessary rigid motions $\Omega$ is a bounded Lipschitz domain  such that $\mathbf 0\in \partial \Omega$ and  $\Omega\cap B_h=\Omega\cap W=S_h$, where $W$ and $S_h$ are defined in \eqref{eq:W} and \eqref{eq:Sh}. Suppose that $v_0$ and $\mathbf u_0$ are defined in \eqref{eq:cgo} and \eqref{eq:u0 cgo}. Let $I_{\Lambda_h,1}$ and $I_{\Lambda_h,2}$ be defined in \eqref{eq:I lambdah}.  Denote
	\begin{equation}\label{eq:int def}
		\begin{split}
			I_1^\pm &=\int_{\Gamma^\pm_h} v_0 (\mathbf u \cdot \nu) \rmd \sigma,\quad I_2^\pm = \int_{\Gamma_h^\pm  }v_0 (\mathbf u \cdot \mathbf p ) \rmd \sigma,\quad I_3^\pm =\int_{\Gamma_h^\pm  }v_0 (\mathbf u \cdot \boldsymbol{\rho } ) \rmd \sigma,\\
			I_4&=\int_{S_h}\kappa^{-1} v v_0\rmd \mathbf x,\quad\quad I_5=\int_{S_h}(\rho_e \mathbf u \cdot \mathbf p) v_0\rmd \mathbf x.
 		\end{split}
	\end{equation}
	Then the following integral identity holds
	\begin{equation}\label{eq:int1 imp}
		\begin{split}
			&\frac{-e^{\mathrm i (\phi-\theta_M ) } \omega^2 \rho_b }{s^2 } \left(  I_1^+ +I_1^- -\frac{1}{\omega^2 \rho_b }I_{\Lambda_h,1}+  I_4\right)+2\mu \Bigg( I_2^+ - e^{\mathrm i (\theta_m-\theta_M )} I_2^-\Bigg) \\
			&\quad +\frac{1}{\bsi s e^{\bsi (\theta_M-\phi )}} I_{\Lambda_h,2} -\frac{\omega^2 }{\bsi s e^{\bsi (\theta_M-\phi )}}	I_5=0.
		\end{split}
	\end{equation}
	
\end{lem}

\begin{proof}
	According to \eqref{eq:GIN1}, using \eqref{eq:v0 delta} and the boundary condition in \eqref{eq:trans3},  one has
	\begin{align}\label{eq:int 233}
	\omega^2 \rho_b 	\int_{S_h}\kappa^{-1} v v_0\rmd \mathbf x=\int_{\Gamma_h^\pm } (v\partial_\nu v_0-\omega^2 \rho_b v_0 (\mathbf u \cdot \nu) )\rmd \sigma+I_{\Lambda_h,1},
	\end{align}
	where $I_{\Lambda_h,1}$ is defined in \eqref{eq:I lambdah}. Due to \eqref{eq:part v0}, one has
	\begin{equation}\label{eq:234 part}
		\partial_\nu v_0 \big|_{\Gamma_h^+}=(\brho \cdot \nu_M )v_0(\mathbf x),\quad \partial_\nu v_0 \big|_{\Gamma_h^-}=(\brho \cdot \nu_m )v_0(\mathbf x),
	\end{equation}
	where $\nu_M$ and $\nu_m$ are given by \eqref{eq:nu pm}. Substituting \eqref{eq:234 part} into \eqref{eq:int 233} and after some calculations, it yields that
	\begin{align}
		0=&\int_{\Gamma_h^+ } v v_0 \rmd \sigma +\frac{\brho \cdot \nu_m}{\brho\cdot \nu_M}\int_{\Gamma_h^- } v v_0 \rmd \sigma -\frac{1}{\brho \cdot \nu_M}\left(\omega^2 \rho_b\int_{\Gamma^\pm_h} v_0 (\mathbf u \cdot \nu) \rmd \sigma-I_{\Lambda_h,1} \right.\notag \\
		&\quad \left. +\omega^2 \rho_b 	\int_{S_h}\kappa^{-1} v v_0\rmd \mathbf x\right). \label{eq:int 236 int}
	\end{align}
	
	Similarly, by virtue of \eqref{eq:u0 lame}, \eqref{eq:green1} and  the boundary condition in \eqref{eq:trans3}, we can deduce that
	\begin{align}\label{eq:int u u_0}
		\omega^2 	\int_{S_h} \rho_e \mathbf u \cdot \mathbf u_0\rmd \mathbf x=\int_{\Gamma_h^\pm } (\mathbf u \cdot T_\nu \mathbf u_0+  v  (\mathbf u_0 \cdot \nu) )\rmd \sigma+I_{\Lambda_h,2},
	\end{align}
	where $I_{\Lambda_h,2}$ is defined in \eqref{eq:I lambdah}. By virtue of \eqref{eq:u0 cgo} and  \eqref{eq:Tu}, one has
	\begin{equation}\label{eq:Tu 238}
		\begin{split}
		\mathbf u_0\cdot \nu_M=(\mathbf p \cdot \nu_M) v_0 , \quad	T_\nu \mathbf u_0 |_{\Gamma_h^+}&=\mu [ ( \brho \cdot \nu_M )\mathbf p+(\mathbf p \cdot \nu_M) \brho ] v_0(\mathbf x),\\
		\mathbf u_0\cdot \nu_m=(\mathbf p \cdot \nu_m) v_0 ,\quad	 T_\nu \mathbf u_0 |_{\Gamma_h^-}&=\mu [ ( \brho \cdot \nu_m )\mathbf p+(\mathbf p \cdot \nu_m) \brho ] v_0(\mathbf x),
		\end{split}
	\end{equation}
	Substituting \eqref{eq:Tu 238} into \eqref{eq:int u u_0}, after some algebraic calculations, it yields that
	\begin{align}
		0=&\int_{\Gamma_h^+ } v v_0 \rmd \sigma +\frac{\mathbf p \cdot \nu_m}{\mathbf p \cdot \nu_M}\int_{\Gamma_h^- } v v_0 \rmd \sigma +\frac{\mu }{\mathbf p \cdot \nu_M}\Bigg ( ( \brho \cdot \nu_M )\int_{\Gamma^+_h} v_0 (\mathbf u \cdot \mathbf p) \rmd \sigma \notag \\ &+( \brho \cdot \nu_m )\int_{\Gamma^-_h} v_0 (\mathbf u \cdot \mathbf p) \rmd \sigma 
		+(\mathbf p \cdot \nu_M)\int_{\Gamma_h^+ } v_0(\mathbf u \cdot \brho)\rmd \sigma  +(\mathbf p \cdot \nu_m)\int_{\Gamma_h^- } v_0(\mathbf u \cdot \brho)\rmd \sigma \notag\\
		&+\mu^{-1}I_{\Lambda_h,2}   -\mu^{-1}\omega^2 	\int_{S_h}(\rho_e \mathbf u \cdot \mathbf p) v_0\rmd \mathbf x\Bigg ). \label{eq:int 239 int}
	\end{align}
	
	In view of \eqref{eq:rho 220a} and \eqref{eq:rho 220c}, we know that
	\begin{equation}\label{eq:240 ratio}
		\frac{\boldsymbol{\rho} \cdot \nu_m }{\boldsymbol{\rho} \cdot \nu_M}=\frac{\mathbf {p} \cdot \nu_m }{\mathbf {p} \cdot \nu_M}=-e^{\mathrm i(\theta_m-\theta_M ) }.
	\end{equation}
Multiplying $-1$ on both sides of \eqref{eq:int 236 int}, then adding it to \eqref{eq:int 239 int}, by virtue of \eqref{eq:240 ratio}, we can deduce that
\begin{align}
&\frac{1}{\brho \cdot \nu_M}\left(	\omega^2 \rho_b\int_{\Gamma^\pm_h} v_0 (\mathbf u \cdot \nu) \rmd \sigma-I_{\Lambda_h,1}+\omega^2 \rho_b 	\int_{S_h}\kappa^{-1} v v_0\rmd \mathbf x\right)+\frac{\mu }{\mathbf p \cdot \nu_M}\Bigg ( ( \brho \cdot \nu_M )\notag \\
&\times \int_{\Gamma^+_h} v_0 (\mathbf u \cdot \mathbf p) \rmd \sigma +( \brho \cdot \nu_m )\int_{\Gamma^-_h} v_0 (\mathbf u \cdot \mathbf p) \rmd \sigma +(\mathbf p \cdot \nu_M)\int_{\Gamma_h^+ } v_0(\mathbf u \cdot \brho)\rmd \sigma  +(\mathbf p \cdot \nu_m) \notag \\
&\times \int_{\Gamma_h^- } v_0(\mathbf u \cdot \brho)\rmd \sigma +\mu^{-1}I_{\Lambda_h,2} -\mu^{-1}\omega^2 	\int_{S_h}(\rho_e \mathbf u \cdot \mathbf p) v_0\rmd \mathbf x\Bigg )=0. \label{eq:241 int}
\end{align}
Multiplying $\dfrac{\mathbf p \cdot \nu_M }{ \boldsymbol{\rho} \cdot \nu_M}$ on both sides of \eqref{eq:241 int}, after some calculations, we have
\begin{align}
&\frac{\mathbf p \cdot \nu_M}{(\brho \cdot \nu_M)^2}\left(	\omega^2 \rho_b\int_{\Gamma^\pm_h} v_0 (\mathbf u \cdot \nu) \rmd \sigma-I_{\Lambda_h,1}+\omega^2 \rho_b 	\int_{S_h}\kappa^{-1} v v_0\rmd \mathbf x\right)+\mu \Bigg (  \int_{\Gamma^+_h} v_0 (\mathbf u \cdot \mathbf p) \rmd \sigma\notag \\
&+ \frac{ \brho \cdot \nu_m }{  \brho \cdot \nu_M }\int_{\Gamma^-_h} v_0 (\mathbf u \cdot \mathbf p) \rmd \sigma +\frac{\mathbf p \cdot \nu_M}{ \brho \cdot \nu_M }\int_{\Gamma_h^+ } v_0(\mathbf u \cdot \brho)\rmd \sigma  +\frac{\mathbf p \cdot \nu_m}{  \brho \cdot \nu_M } \int_{\Gamma_h^- } v_0(\mathbf u \cdot \brho)\rmd \sigma \Bigg )\notag \\
&+ \frac{1}{\brho \cdot \nu_M} I_{\Lambda_h,2} -\frac{\omega^2 }{\brho \cdot \nu_M}	\int_{S_h}(\rho_e \mathbf u \cdot \mathbf p) v_0\rmd \mathbf x=0. \label{eq:242 int}
\end{align}
In view of \eqref{eq:rho 220a} and \eqref{eq:rho 220c}, it can be directly to verify that

		\begin{alignat}{2}
		\frac{\mathbf p \cdot \nu_M}{(\brho \cdot \nu_M)^2}&=\frac{-e^{\mathrm i (\phi-\theta_M ) }}{s^2 },&\hspace{10pt} \frac{ \brho \cdot \nu_m }{  \brho \cdot \nu_M }&=-e^{\mathrm i (\theta_m-\theta_M ) },\nonumber \\
		\frac{\mathbf p \cdot \nu_m}{  \brho \cdot \nu_M }&=-\frac{e^{\mathrm i(\theta_m-\theta_M )}}{\mathrm i s }, &\hspace{10 pt} \frac{\mathbf p \cdot \nu_M}{  \brho \cdot \nu_M }&=\frac{1}{\mathrm i s }\label{eq:243 frac}.
	\end{alignat}
Substituting \eqref{eq:243 frac} into \eqref{eq:242 int}, we can obtain
\begin{equation}\label{eq:int1 imp pre}
		\begin{split}
			&\frac{-e^{\mathrm i (\phi-\theta_M ) } \omega^2 \rho_b }{s^2 } \left(  I_1^+ +I_1^- -\frac{1}{\omega^2 \rho_b }I_{\Lambda_h,1}+  I_4\right)+\mu \Bigg( I_2^+ - e^{\mathrm i (\theta_m-\theta_M )} I_2^-+\frac{1}{\mathrm i s} I_3^+ \\
			&\quad -\frac{e^{\mathrm i(\theta_m-\theta_M )}}{\mathrm i s } I_3^-\Bigg)+\frac{1}{\bsi s e^{\bsi (\theta_M-\phi )}} I_{\Lambda_h,2} -\frac{\omega^2 }{\bsi s e^{\bsi (\theta_M-\phi )}}	I_5=0,
		\end{split}
	\end{equation}
	where $I_\ell^\pm $ is defined in \eqref{eq:int def}, $\ell=1,\ldots,5$. Recall that $\boldsymbol{\rho} $ and $\mathbf p$ are defined in  \eqref{eq:cgo} and \eqref{eq:u0 cgo}, respectively. By \eqref{eq:248 rho p}, it yields that
	\begin{equation}\label{eq:249 I2I3}
		\frac{I_3^\pm }{\mathrm  i s} = I_2^\pm.
	\end{equation}
	Substituting \eqref{eq:249 I2I3} into \eqref{eq:int1 imp pre}, we derive  \eqref{eq:int1 imp} by performing straightforward calculations.
\end{proof}


Under certain regularity assumptions of the AE transmission eigenfunction ($v,\mathbf u$) to \eqref{eq:trans3} near a corner, Lemma \ref{lem:alpha} and Lemma \ref{lem:estimate} provide crucial estimates for integrals in  \eqref{eq:int1 imp}, which will be utilized in the proof of Theorem \ref{thm:thm21}.

\begin{lem}\label{lem:alpha}
	Suppose that the transmission eigenfunction $( v, \mathbf u)\in  H^1(\Omega) \times H^1(\Omega)^2 $ to \eqref{eq:trans2}  satisfies \eqref{eq:trans3},  $\mathbf u =(u_\ell)_{\ell=1}^2\in  C^{1,\alpha_1}({\overline{S_h} })^2 $,  and $ v \in C^{\alpha_2} (\overline{ S_h} )$ for $\alpha_1,\,\alpha_2 \in (0,1)$. 
Then we have
\begin{align}
	I_4&=\kappa^{-1} v(\mathbf 0) \frac{\bsi e^{2\bsi
		\phi }}{ 2 s^{2}}\left(e^{-2\bsi \theta_M} - e^{-2\bsi \theta_m}\right)+	\Oh\left(\frac{1}{s^{2+\alpha_2}}\right),\label{eq:I4 est} \\
I_5&=\rho_e (\mathbf u(\mathbf 0 ) \cdot \mathbf p) \frac{\bsi e^{2\bsi
		\phi }}{ 2 s^{2}}\left(e^{-2\bsi \theta_M} - e^{-2\bsi \theta_m}\right)+	\Oh\left(\frac{1}{s^{2+\alpha_1}}\right),\label{eq:I5 est}
	\end{align}
as $s\rightarrow +\infty.  $ 
\end{lem}
\begin{proof}
 Due to $v \in C^{\alpha_2}({\overline{S_h} })$
 and $\mathbf u \in C^{1,\alpha_1}({\overline{S_h} })^2$, we have the expansions
 \begin{equation}\label{eq:f1f2 exp}
 \begin{split}
 	v(\mathbf x)&=v(\mathbf 0) +\delta v,\  |\delta v | \leq \|v\|_{C^{\alpha_2}(\overline{S_h} ) } |\mathbf x|^{\alpha_2},\\
 	\mathbf u(\mathbf x)&=\mathbf u(\mathbf 0) + r \nabla \mathbf u(\mathbf 0)\boldsymbol{\tau}+ \delta \mathbf u,\  |\delta \mathbf u | \leq \|\mathbf u\|_{C^{1,\alpha_1}(\overline{S_h} )^2 } |\mathbf x|^{1+\alpha_1} ,
 	 \end{split}
 \end{equation}
 where $\mathbf x=r\boldsymbol{\tau }\in \mathbb R^2$ with $\boldsymbol{\tau }=(\cos\theta, \sin \theta )$ and $r\in \mathbb R_+\cup \{0\}$.
 Substituting \eqref{eq:f1f2 exp} into the definition of $I_4$ and $I_5$, it arrives at
\begin{align}\label{eq:I45 265}
I_4&=\kappa^{-1} v(\mathbf 0)\int_{S_h} v_0\rmd \mathbf x+ \kappa^{-1} \int_{S_h} \delta v v_0\rmd \mathbf x,\\
I_5&=\rho_e (\mathbf u(\mathbf 0 ) \cdot \mathbf p)\int_{S_h} v_0\rmd \mathbf x+ \rho_e\int_{S_h} r\mathbf p^\top  \nabla \mathbf u (\mathbf 0) \boldsymbol{\tau} \rmd \mathbf x + \rho_e\int_{S_h}(\delta \mathbf u \cdot \mathbf p) v_0\rmd \mathbf x.\notag
\end{align}




By Cauchy-Schwarz inequality, \eqref{eq:f1f2 exp}, and \eqref{eq:v0 alpha}, we have
\begin{align}\label{eq:I45 est 266}
	\left| \int_{S_h} \delta v v_0\rmd \mathbf x \right| &\leq \|v\|_{C^{\alpha_2}(\overline{S_h} ) } \frac{(\theta_M-\theta_m) \Gamma(\alpha_2 +2) }{ (\delta s)^{\alpha_2 +2}} + \Oh(e^{-s\delta h/2} ),\\
	\left| \int_{S_h}( \delta \mathbf  u \cdot \mathbf p) v_0\rmd \mathbf x \right| &\leq 2 \|\mathbf u\|_{C^{1,\alpha_1}(\overline{S_h} )^2 } \frac{(\theta_M-\theta_m) \Gamma(\alpha_1 +3) }{ (\delta s)^{\alpha_1 +3}} + \Oh(e^{-s\delta h/2} ), \notag
\end{align}
as $s\rightarrow +\infty.  $ 
By virtue of \eqref{eq:1r1} and \eqref{eq:v0sh}, one can deduce that
\begin{align}\label{eq:Sv}
		&\int_{S_h} r\mathbf p^\top  \nabla \mathbf u (\mathbf 0) \boldsymbol{\tau} \rmd \mathbf x  = \Oh(\frac{1}{s^3}) ,\notag\\
&	\int_{S_h} v_0\rmd \mathbf x = \int_{\theta_m }^{\theta_M } \int_0^h r e^{r\boldsymbol{\rho} \cdot  \boldsymbol \tau } \rmd r \rmd \theta = \frac{\bsi e^{2\bsi
			\phi }}{ 2 s^{2}}\left(e^{-2\bsi \theta_M} - e^{-2\bsi \theta_m}\right),
	\end{align}
as $s\rightarrow +\infty.  $ 

Substituting \eqref{eq:I45 est 266}, \eqref{eq:Sv} into \eqref{eq:I45 265}, then we obtain \eqref{eq:I4 est}  and \eqref{eq:I5 est} .
\end{proof}

\begin{lem}\label{lem:estimate}
	Suppose that the transmission eigenfunction $( v, \mathbf u)\in  H^1(\Omega) \times H^1(\Omega)^2 $ to \eqref{eq:trans2} satisfying \eqref{eq:trans3},  Moreover, let $\mathbf u$ belong to $ C^{1,\alpha_1}({\overline{S_h} })^2 $ for $\alpha_1 \in (0,1)$. Then the following integral estimations hold 	
	\begin{align}
	I_1^+&= - (\mathbf u(\mathbf 0) \cdot \nu_M)\frac{e^{\mathrm i (\phi-\theta_M )}}{s} 	+ \nu_M ^\top \nabla \mathbf u(\mathbf 0)\boldsymbol{\tau}_M \frac{e^{2\mathrm i (\phi-\theta_M )}}{s^2} +\Oh\left(\frac{1}{s^{2+\alpha_1}}\right) ,\label{eq:I1+ est}\\
	 I_1^-&=- (\mathbf u(\mathbf 0) \cdot \nu_m)\frac{e^{\mathrm i (\phi-\theta_m )}}{s} +\nu_m ^\top \nabla \mathbf u(\mathbf 0)\boldsymbol{\tau}_m \frac{e^{2\mathrm i (\phi-\theta_m )}}{s^2} +	\Oh\left(\frac{1}{s^{2+\alpha_1}}\right), \label{eq:I1- est}\\
		I_2^+&=-(\mathbf u (\mathbf 0)\cdot \mathbf p )\frac{e^{\mathrm i (\phi-\theta_M )}}{s}   + \mathbf p ^\top \nabla \mathbf u(\mathbf 0)\boldsymbol{\tau}_M \frac{e^{2\mathrm i (\phi-\theta_M )}}{s^2} +	\Oh\left(\frac{1}{s^{2+\alpha_1 }}\right),  \label{eq:I2+ est}\\
		I_2^-&=-(\mathbf u (\mathbf 0)\cdot \mathbf p )\frac{e^{\mathrm i (\phi-\theta_m )}}{s}   + \mathbf p ^\top \nabla \mathbf u(\mathbf 0)\boldsymbol{\tau}_m \frac{e^{2\mathrm i (\phi-\theta_m )}}{s^2}   +	\Oh\left(\frac{1}{s^{2 +\alpha_1 }}\right), \label{eq:I2- est}  
			\end{align}
			as $s\rightarrow +\infty$. 
\end{lem}

\begin{proof}
According to $\mathbf u \in  C^{1,\alpha_1}({\overline{S_h} })^2 $, one directly obtain that
\begin{equation}\label{eq: u c1}
	\mathbf u(\mathbf x)=\mathbf u(\mathbf 0 )+r \nabla \mathbf u(\mathbf 0)\boldsymbol{\tau}  +\delta \mathbf u, \quad |\delta \mathbf u | \leq \|\mathbf u\|_{C^{1,\alpha _1} (\overline{S_h} )^2 } |\mathbf x|^{1+\alpha_1},  
\end{equation}
where $\mathbf x=r\boldsymbol{\tau }\in \mathbb R^2$ with $\boldsymbol{\tau }=(\cos\theta, \sin \theta )$ and $r\in \mathbb R_+\cup \{0\}$. Substituting \eqref{eq: u c1} into $I_2^\pm $ defined in \eqref{eq:int def}, we have
\begin{align}
	I_2^+ &=(\mathbf u (\mathbf 0)\cdot \mathbf p ) \int_{\Gamma_h^+  }v_0 (\mathbf x) \rmd \sigma + \mathbf p ^\top \nabla \mathbf u(\mathbf 0)\boldsymbol{\tau}_M \int_{0}^h r e^{r \boldsymbol{\rho} \cdot \boldsymbol{ \tau}_M } \rmd r  +	r_{I_2^+},\label{eq:I2+ ex} \\
	I_2^-&=(\mathbf u (\mathbf 0)\cdot \mathbf p ) \int_{\Gamma_h^-  }v_0 (\mathbf x) \rmd \sigma + \mathbf p ^\top \nabla \mathbf u(\mathbf 0)\boldsymbol{\tau}_m \int_{0}^h r e^{r \boldsymbol{\rho} \cdot \boldsymbol{ \tau}_m } \rmd r +	r_{I_2^-}, \label{eq:I2- ex} 
\end{align}
where
\begin{equation*}
	r_{I_2^+}=\int_{\Gamma_h^+ } (\delta \mathbf u \cdot \mathbf p) v_0(\mathbf x) \rmd \sigma, \quad 	r_{I_2^-}=\int_{\Gamma_h^- } (\delta \mathbf u \cdot \mathbf p) v_0(\mathbf x) \rmd \sigma.
\end{equation*}
Using \eqref{eq:v0 gamma} and \eqref{eq:rho 220b}, one can directly has
\begin{equation}\label{eq:v0 int}
		\int_{\Gamma_h^+} v_0(\mathbf x) \rmd \sigma=-\frac{e^{\mathrm i (\phi-\theta_M ) }}{s}+\Oh (e^{-s\delta h/2 } ),\quad \int_{\Gamma_h^-} v_0(\mathbf x) \rmd \sigma=-\frac{e^{\mathrm i (\phi-\theta_m ) }}{s}+\Oh (e^{-s\delta h/2 } ),
	\end{equation}
	as $s \rightarrow +\infty$, where $\phi$ is the polar angle of $\mathbf d$ defined in \eqref{eq:d cond}. Similarly, according to \eqref{eq:v0 gamma} and \eqref{eq:rho 220b}, we know that
\begin{align}\label{eq:I2 253 est}
	\int_0^h r^\ell e^{r \boldsymbol{\rho} \cdot \boldsymbol{\tau}_M } \rmd r&=\frac{\ell !e^{(\ell+1) \mathrm i (\phi-\theta_M )}}{(-s)^{\ell+1}} +\Oh(e^{-sh \delta/2}), \\
	\int_0^h r^\ell e^{r \boldsymbol{\rho} \cdot \boldsymbol{\tau}_m } \rmd r &=\frac{\ell ! e^{(\ell+1)\mathrm i (\phi-\theta_m )}}{(-s)^{\ell+1}} +\Oh(e^{-sh \delta/2}) , \quad \ell \in \mathbb{N} \notag
\end{align}
as $s\rightarrow +\infty$. Using the Cauchy-Schwarz inequality, \eqref{eq:v0 gamma}, and combining \eqref{eq:rho 220b}, \eqref{eq:p ineq} with \eqref{eq: u c1}, one can deduce that
\begin{align}\label{eq:I2 est 254}
	\left|r_{I_2^+}\right| &\leq  2 \|\mathbf u\|_{C^{1,\alpha_1 } (\overline{S_h} )   } \int_0^h r^{1+\alpha_1 } \,e^{r \boldsymbol{\rho} \cdot \boldsymbol{\tau}_M } \rmd r=\frac{2 \|\mathbf u\|_{C^{1, \alpha_1 } (\overline{S_h} )   }  e^{(2+\alpha_1 )\mathrm i(\phi-\theta_M )}}{(-s)^{2+\alpha_1 }} +\Oh(e^{-sh \delta/2}),\\
	\left|r_{I_2^-}\right| &\leq  2 \|\mathbf u\|_{C^{1, \alpha_1 } (\overline{S_h} )   } \int_0^h r^{1+\alpha_1 }\, e^{r \boldsymbol{\rho} \cdot \boldsymbol{\tau}_m } \rmd r=\frac{2 \|\mathbf u\|_{C^{1, \alpha_1 } (\overline{S_h} )   }  e^{(2+\alpha_1)\mathrm i (\phi-\theta_m)}}{(-s)^{2 +\alpha_1 }} +\Oh(e^{-sh \delta/2}), \notag
\end{align}
as $s \rightarrow +\infty $. Substituting \eqref{eq:v0 int},  \eqref{eq:I2 253 est} and \eqref{eq:I2 est 254}  into \eqref{eq:I2+ ex} and  \eqref{eq:I2- ex} respectively, we derive \eqref{eq:I2+ est}  and \eqref{eq:I2- est}.

	Substituting \eqref{eq: u c1} into the expression of $I_1^\pm$ defined in \eqref{eq:int def}, it arrives at
		\begin{align}\label{eq:I1 ex}
		I_1^+&= (\mathbf u(\mathbf 0) \cdot \nu_M) \int_{\Gamma^+_h} v_0 (\mathbf x) \rmd \sigma +\nu_M ^\top \nabla \mathbf u(\mathbf 0)\boldsymbol{\tau}_M \int_{0}^h r e^{r \boldsymbol{\rho} \cdot \boldsymbol{ \tau}_M } \rmd r+ r_{I_1^+},\\
		 I_1^- &= (\mathbf u(\mathbf 0) \cdot \nu_m) \int_{\Gamma^-_h} v_0 (\mathbf x) \rmd \sigma +\nu_m ^\top \nabla \mathbf u(\mathbf 0)\boldsymbol{\tau}_m \int_{0}^h r e^{r \boldsymbol{\rho} \cdot \boldsymbol{ \tau}_m } \rmd r+ r_{I_1^-}, \notag
	\end{align}
	where
	\begin{equation*}
		\begin{split}
			r_{I_1^\pm }=\int_{\Gamma_h^\pm  } v_0(\mathbf x) (\delta \mathbf u \cdot \nu  ) \rmd \sigma.
		\end{split}
	\end{equation*}
By virtue of Cauchy-Schwarz inequality, \eqref{eq:v0 gamma al},  and \eqref{eq: u c1}, we obtain that
	\begin{align}\label{eq:I1 est 260}
	\left|r_{I_1^\pm }\right| &\leq   \|\mathbf u\|_{C^{1,\alpha_1} (\overline{S_h} )   } \int_{\Gamma_h^\pm  } | v_0(\mathbf x)| |\mathbf x|^{1+\alpha_1 }\rmd \sigma\leq \frac{ \Gamma(\alpha_1+2) }{ (\delta s)^{\alpha_1+2}} + \Oh(e^{-s\delta h/2} )
\end{align}
as $s\rightarrow + \infty. $ Substituting \eqref{eq:v0 int} and \eqref{eq:I2 253 est} into \eqref{eq:I1 ex}, and in virtue of \eqref{eq:I1 est 260}, it's sufficient to attain \eqref{eq:I1+ est} and  \eqref{eq:I1- est}.
\end{proof}


\begin{rem}

The regularity assumptions of the transmission eigenfunction $(v,\mathbf u)$ to \eqref{eq:trans3} in Lemma \ref{lem:alpha} and  Lemma \ref{lem:estimate} are crucial for analyzing the vanishing properties of $(v,\mathbf u)$ at a corner, as stated in Theorem \ref{thm:thm21}. In Section \ref{sec:ip}, we shall explore the inverse problem \eqref{eq:so1} related to \eqref{eq:transea1} by utilizing the results from Theorem \ref{thm:thm21}. Specifically, Lemma \ref{lem:alpha} and Lemma \ref{lem:estimate} require that $\mathbf u\in C^{1,\alpha_1}({\overline{S_h} })^2 $ and $ v\in C^{\alpha_2}(\overline{ S_h})$, where $\alpha_1,\,\alpha_2\in (0,1)$. These regularity conditions can be met when addressing the inverse problem \eqref{eq:so1}, as discussed in detail in Theorem \ref{thm:inverse1}.

	

\end{rem}

The lemma below describes the properties of the coefficient matrix in a quadratic form related to a corner. These characteristics are essential for the proof of Theorem \ref{thm:thm21}.

\begin{lem}\label{lem:conclusion}
Assume that $\phi$ is the polar angle of $\mathbf d$ defined in \eqref{eq:cgo}, where $\mathbf d$ fulfills \eqref{eq:d cond}.  Let $\boldsymbol{A}=(a_{ij})_{i,j=1}^2 \in \mathbb R^{2\times 2 }$ satisfying
	\begin{equation}\label{eq:eqn A}
		\mathbf p^\top \boldsymbol{A} \left( e^{2\mathrm i (\phi-\theta_M )} \boldsymbol{\tau}_M - e^{\mathrm i (2\phi-\theta_m - \theta_M )}\boldsymbol{\tau}_m  \right)=0,
	\end{equation}
	where   $\mathbf p$, $\boldsymbol{\tau}_m  $, and $\boldsymbol{\tau}_M$ are defined in \eqref{eq:u0 cgo} and \eqref{eq:taumM} respectively. If $W$ is strictly convex, where $W$ is defined in \eqref{eq:W} with the opening angle $\theta_M-\theta_m$,
	one has
	$$
	a_{11}=a_{22} \mbox{ and } a_{12}+a_{21}=0.
	$$
\end{lem}

\begin{proof}
By virtue of \eqref{eq:taumM}, it can be deduce that
\begin{align}\label{eq:270 taum}
	e^{2\mathrm i (\phi-\theta_M )} \boldsymbol{\tau}_M - e^{\mathrm i (2\phi-\theta_m - \theta_M )}\boldsymbol{\tau}_m  &=e^{2\mathrm i (\phi -\theta_M ) } \left(\boldsymbol{\tau}_M-e^{\mathrm i ( \theta_M -\theta_m )}\boldsymbol{\tau}_m \right) \\
	&=e^{2\mathrm i (\phi -\theta_M )-\mathrm i\theta_m } \left(e^{\mathrm i \theta_m  }\boldsymbol{\tau}_M-e^{\mathrm i  \theta_M }\boldsymbol{\tau}_m \right)\notag \\
	&=e^{2\mathrm i (\phi -\theta_M )-\mathrm i\theta_m }\sin(\theta_M-\theta_m )\begin{bmatrix}
		-\mathrm i \\ 1
	\end{bmatrix}. \notag
\end{align}
According to \eqref{eq:d cond} and \eqref{eq:cgo}, after some calculations, one has
\begin{align}\label{eq:271 p}
	\mathbf p= e^{-\mathrm i \phi} \begin{bmatrix}
		-\mathrm i \\ 1
	\end{bmatrix}.
\end{align}
Since $W$ is strictly convex, one has $0<\theta_M-\theta_m< \pi$, which imples that  $\sin(\theta_M-\theta_m )>0$. Therefore, substituting \eqref{eq:270 taum} and \eqref{eq:271 p} into \eqref{eq:eqn A}, it yields that
\begin{align}\label{eq:A 2370}
	\begin{bmatrix}-\mathrm i& 1\end{bmatrix} \boldsymbol{A} \begin{bmatrix}
		-\mathrm i\\ 1
	\end{bmatrix}=0.
\end{align}
Comparing the real and imaginary part of both sides of \eqref{eq:A 2370} separately, we finish the proof.
\end{proof}



Now we are ready to proceed to the proof of our main theorem in this section. 

\begin{thm}\label{thm:thm21}
	Assume that a bounded Lipschitz domain $\Omega\Subset\mathbb{R}^{2}$ contains a corner. Let  $S_h = \Omega\cap W$ with $h\ll 1$, where $S_h \Subset \Omega$ and $W$ is  a sector defined in \eqref{eq:W}. Without loss of generality, by applying rigid motions if necessary, we can assume that the vertex of the corner $S_h$ is at $\mathbf 0 \in \partial \Omega $. Suppose that $\lambda, \mu, \rho_e, \rho_b$ and $\kappa$ are real-valued constants, with $\lambda$ and $\mu$ satisfying the condition in \eqref{eq:cond1}. Let $( v, \mathbf u)\in  H^1(\Omega) \times H^1(\Omega)^2 $ be a pair of transmission eigenfunctions for the problem \eqref{eq:trans2} associated with $\omega \in \mathbb{R}_+$ that satisfies \eqref{eq:trans3}.
  \begin{enumerate}
	\item [(a)] Assume that 
	\begin{equation}\label{eq:thm21 assump}
	\mathbf u \in C^{1,\alpha_1 }({\overline{S_h} })^2 \quad {\mbox  and }\quad  v \in C^{\alpha_2} (\overline{ S_h} )
	\end{equation}
	 for $\alpha_1,\, \alpha_2\in (0,1)$.  
Then we have
\begin{equation}\label{eq:thm21}
\nabla^s \mathbf u (\mathbf 0)=\partial_1 u_1 (\mathbf 0)	I,
\end{equation}
where $\nabla^s \mathbf u$ is the strain tensor of $\mathbf u$ defined in 
\eqref{eq:traction1} and $I$ is the identity matrix;
\item [(b)] Under the assumption in $(a)$, one has
\begin{align}\label{eq:thm21 v0}
	v(\mathbf 0)=-2(\lambda+\mu ) \partial_1 u_1(\mathbf 0 );
\end{align}
\item [(c)] Moreover, under the assumption in $(b)$ and suppose that $\partial_1 u_1(\mathbf{0}) = 0$, then we have
\begin{equation}\label{eq:v0=0}
	v(\mathbf{0})=0;
\end{equation}
\item [(d)]Furthermore, under the assumptions in $(b)$ and $(c)$,if $\partial_2 u_1(\mathbf{0}) = 0$, then it holds that
\begin{equation}\label{eq:gra u0=0}
	\nabla \mathbf{u}(\mathbf{0})=\mathbf{0}.
\end{equation}
\end{enumerate}
\end{thm}

\begin{proof}[Proof of Theorem~\ref{thm:thm21}] 
Given that the corner is strictly convex, there exists a unit vector $\mathbf d$ satisfying the condition in \eqref{eq:d cond}. Consequently, the complex geometric optics (CGO) solutions $v_0$ and $\mathbf u_0$ can be constructed in the forms specified by \eqref{eq:cgo} and \eqref{eq:u0 cgo}, respectively.

Let $I_1^\pm$, $I_2^\pm$, $I_4$, $I_5$ and $I_{\Lambda_h,\ell }$ ($\ell=1,2$) be defined in \eqref{eq:I lambdah} and \eqref{eq:int def} respectively. According to Lemma \ref{lem:28 int}, we have the integral identity \eqref{eq:int1 imp}.  By \eqref{eq:I1+ est} and \eqref{eq:I1- est}, we know that
\begin{align}
	-\frac{e^{\mathrm i (\phi-\theta_M ) } \omega^2 \rho_b }{s^2 } \left(I_1^++I_1^-\right) &= \omega^2 \rho_b\left( (\mathbf u(\mathbf 0) \cdot \nu_M)\frac{e^{2\mathrm i (\phi-\theta_M )}}{s^3} +	(\mathbf u(\mathbf 0) \cdot \nu_m)\frac{e^{\mathrm i (2\phi-\theta_m- \theta_M )}}{s^3}\right) \notag
	\\
	&\quad +\Oh\left(\frac{1}{s^{3+\alpha_1}}\right), \label{eq:275 int}
	\end{align}
as $s\rightarrow +\infty$.  According to \eqref{eq:I2+ est} and \eqref{eq:I2- est}, it is readily to know that
\begin{align}\label{eq:276 int}
	&I_2^+ - e^{\mathrm i (\theta_m-\theta_M )} I_2^-=\frac{ \mathbf p ^\top \nabla \mathbf u(\mathbf 0)}{s^2} \left( e^{2\mathrm i (\phi-\theta_M )} \boldsymbol{\tau}_M - e^{\mathrm i (2\phi-\theta_m - \theta_M )}\boldsymbol{\tau}_m  \right)  +	\Oh\left(\frac{1}{s^{2 +\alpha _1}}\right) ,
\end{align}
as $s\rightarrow +\infty$.  Using \eqref{eq:I4 est} and \eqref{eq:I5 est}, we can derive that
\begin{align}\label{eq:277 int}
	&-\frac{e^{\mathrm i (\phi-\theta_M ) } \omega^2 \rho_b }{s^2 } I_4 -\frac{\omega^2 }{\bsi s e^{\bsi (\theta_M-\phi )}}	I_5\\ \notag 
	=& -\frac{\bsi\omega^2  e^{\bsi (3\phi-\theta_M)  }}{ 2 s^{3}}\left(e^{-2\bsi \theta_M} - e^{-2\bsi \theta_m}\right)\left[\frac{\rho_b \kappa^{-1} v(\mathbf 0)}{s} - \mathrm i \rho_e (\mathbf u(\mathbf 0 ) \cdot \mathbf p)\right] 
	 +	\Oh\left(\frac{1}{s^{3+\alpha_1}}\right),
\end{align}
as $s\rightarrow +\infty$.

Substituting \eqref{eq:275 int}-\eqref{eq:277 int} into \eqref{eq:int1 imp}, then multiplying $s^2$ on both sides, after rearranging terms,  we have
\begin{align}
	&2 \mu \mathbf p ^\top \nabla \mathbf u(\mathbf 0) \left( e^{2\mathrm i (\phi-\theta_M )} \boldsymbol{\tau}_M - e^{\mathrm i (2\phi-\theta_m - \theta_M )}\boldsymbol{\tau}_m  \right)
 =\frac{\bsi\omega^2  e^{\bsi (3\phi-\theta_M)  }}{ 2 s}\left(e^{-2\bsi \theta_M} - e^{-2\bsi \theta_m}\right)\notag \\
	&\quad \times \left[\frac{\rho_b \kappa^{-1} v(\mathbf 0)}{s} -\mathrm i \rho_e (\mathbf u(\mathbf 0 ) \cdot \mathbf p)\right]-\omega^2 \rho_b\left[  (\mathbf u(\mathbf 0) \cdot \nu_m) e^{\mathrm i (\theta_m- \theta_M )} +(\mathbf u(\mathbf 0) \cdot \nu_M) \right] \frac{e^{\mathrm i (2\phi-\theta_m- \theta_M )}}{s}\notag \\
	&\quad +e^{\mathrm i (\phi- \theta_M )} (\mathrm i s I_{\Lambda_h,2}-I_{\Lambda_h,1}) +\Oh\left(\frac{1}{s^{\alpha_2}}\right) , \label{eq:int eq final}
\end{align}
as $s \rightarrow + \infty$. Since $\mu>0$, using \eqref{eq:I arc1} and \eqref{eq:I arc2}, letting $s\rightarrow +\infty$ in \eqref{eq:int eq final}, it arrives that
\begin{align}\label{eq:gradient u=0}
	\mathbf p ^\top \nabla \mathbf u(\mathbf 0) \left( e^{2\mathrm i (\phi-\theta_M )} \boldsymbol{\tau}_M - e^{\mathrm i (2\phi-\theta_m - \theta_M )}\boldsymbol{\tau}_m  \right)=0,
\end{align}
where $\mathbf p$, $\boldsymbol{\tau}_m$, and $\boldsymbol{\tau}_M$ are defined in \eqref{eq:u0 cgo} and \eqref{eq:taumM}, respectively.  Given that $\mathbf u$ is real-valued and the corner is strictly convex, we combine \eqref{eq:gradient u=0} with Lemma \ref{lem:conclusion} to prove \eqref{eq:thm21}.

Recall that $\nu_m$ is the exterior unit normal vector to $\Gamma_h^-$, as defined in the \eqref{eq:nu pm}. 
Since $\mathbf u \in C^{1,\alpha_1 }({\overline{S_h} })^2 $, substituting \eqref{eq:nu pm} into \eqref{eq:traction1} yields
\begin{align}
T_\nu\mathbf{u} & =\lambda(\nabla\cdot\mathbf{u})\nu_m+2\mu(\nabla^s\mathbf{u})\nu_m \notag\\
& = \lambda ( \partial_1 u_1 + \partial_2 u_2) \left(
  \begin{array}{c}
    \sin\theta_m \\ -\cos\theta_m
  \end{array}
  \right)
  + \mu \left(
  \begin{array}{c}
  2 \partial_1 u_1 \cdot  \sin \theta_m - (\partial_1 u_2 + \partial_2 u_1)\cdot \cos\theta_m\\
  -2 \partial_2 u_2 \cdot\cos \theta_m + (\partial_1 u_2 + \partial_2 u_1) \cdot \sin\theta_m
  \end{array}
  \right)\notag\\
 & = \left(
 \begin{array}{c}
   \lambda ( \partial_1 u_1 + \partial_2 u_2) \cdot \sin \theta_m + 2 \mu \partial_1 u_1 \cdot \sin\theta_m - \mu (\partial_1 u_2 + \partial_2 u_1)\cdot \cos\theta_m\\
   -\lambda ( \partial_1 u_1 + \partial_2 u_2) \cdot \cos \theta_m - 2 \mu \partial_2 u_2 \cdot \cos\theta_m + \mu (\partial_1 u_2 + \partial_2 u_1)\cdot \sin\theta_m
 \end{array}
 \right) \notag \\
   & = \left(
  \begin{array}{c}
  - \mu (\partial_1 u_2 + \partial_2 u_1)\\
  	-\lambda ( \partial_1 u_1 + \partial_2 u_2)  - 2 \mu \partial_2 u_2 
  \end{array}
  \right).\label{eq:Tum1}
\end{align}

Next, we take $\mathbf{x}=\mathbf{0}$ and substitute \eqref{eq:thm21} into the last boundary condition in \eqref{eq:trans3},  and using \eqref{eq:Tum1}, we can readily derive \eqref{eq:thm21 v0}. Combining \eqref{eq:thm21} with \eqref{eq:thm21 v0} allows us to directly obtain \eqref{eq:v0=0} and \eqref{eq:gra u0=0}.
\end{proof}

\section{Local geometrical properties near corners of acoustic-elastic transmission eigenfunctions: three dimensional case}\label{sec:3}


In this section, we investigate the geometrical property of the transmission eigenfunction pair $(v,\mathbf{u}) $ to \eqref{eq:trans2} close to the corner in $\mathbb R^3$.  While it is possible to generalize to a corner similar to those in the 2D case, we focus specifically on a 3D corner defined by $S_h \times (-M,M)$, where $S_h$ is defined in \eqref{eq:Sh} and $M \in \mathbb{R}_+$. $S_h \times (-M,M)$. This defines an edge singularity. Assume that $\Omega$ is a Lipschitz domain in $\mathbb{R}^3$ containing $\mathbf{0}$ on its boundary and has a 3D edge corner. Consider $\mathbf{0} \in \mathbb{R}^2$ as the vertex of $S_h$ and let $x_3 \in (-M,M)$. The point $(\mathbf{0}, x_3)$ then represents an edge point of $S_h \times (-M,M)$. 

\par \qquad \\[-6.0em]
\begin{figure}[h]
\centering
\subfigure{\includegraphics[width=0.38\textwidth]{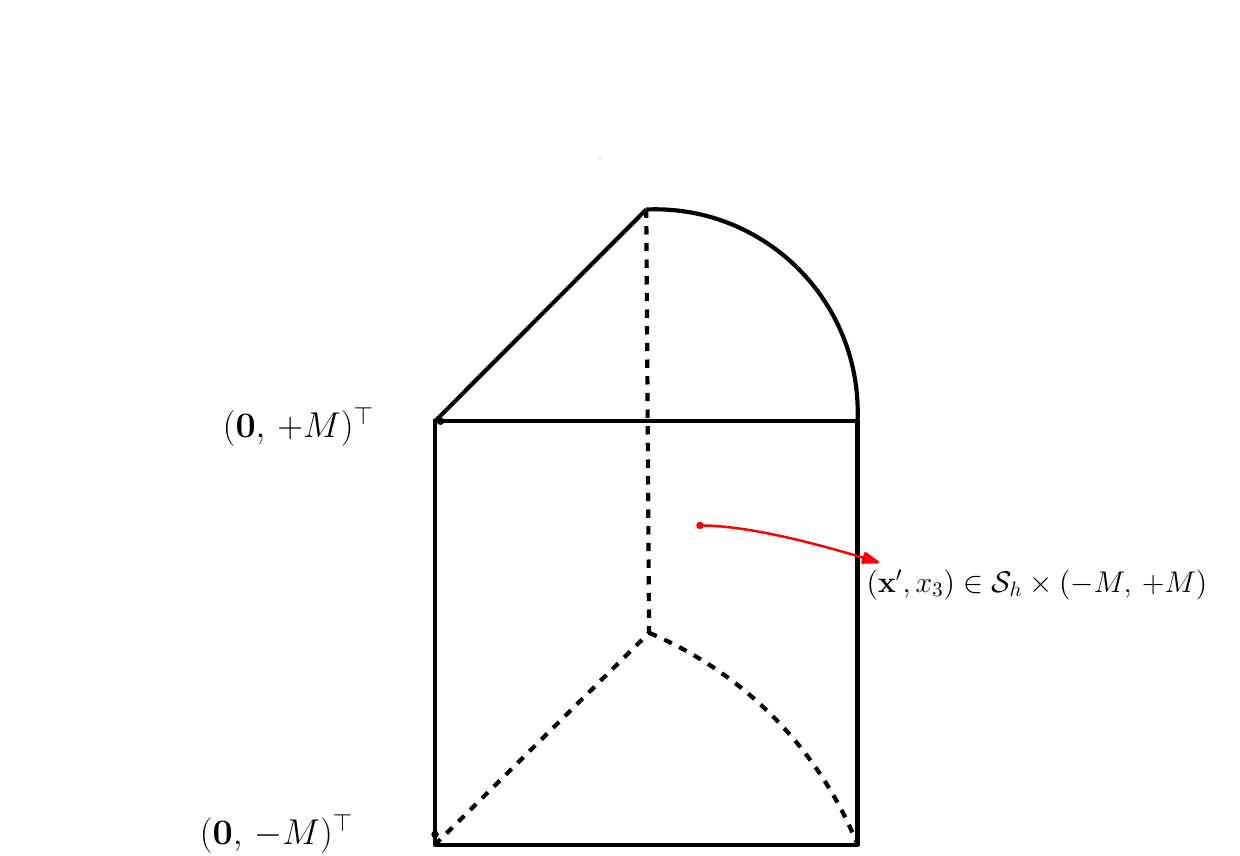}}\\
\caption{Schematic illustration of a 3D edge corner. \label{fig:3}}
\end{figure}

To investigate the vanishing property of the transmission eigenfunction pair $(v,\mathbf{u})$ for \eqref{eq:trans2} at a 3D edge corner, we utilize the CGO solution $\mathbf{u}(\mathbf{x})$ introduced in Lemma \ref{lem:u0}. To facilitate this analysis, we define a dimension reduction operator as follows:

\begin{defn}\label{def:RO}
	Let $S_h \Subset \mathbb{R}^2$ be defined in \eqref{eq:Sh}, $M>0$. For any point $x_3 \in (-M,M)$, suppose that $\psi \in C_0^{\infty}((x_3-L,x_3+L))$ is a non-negative function and $\psi \not \equiv 0 $, where $L$ is sufficiently small such that $(x_3-L,x_3+L) \Subset (-M,M)$, and write $\mathbf{x}=(\mathbf{x}^{\prime},x_3) \in \mathbb{R}^3$, $\mathbf{x}^{\prime} \Subset \mathbb{R}^2$. The dimension reduction operator $\mathcal{R}$ is defined by
	\begin{equation}\label{eq:RO}
		\mathcal{R}(\mathbf{g})(\mathbf{x^{\prime}}) = \int_{x_3-L}^{x_3+L} \psi(x_3) \mathbf{g}(\mathbf{x^{\prime}}, x_3) \rmd x_3,
		\end{equation}
	where  $\mathbf{g}$ is a given function defined in the domain $S_h \times (-M,M)$ and $\mathbf{x^{\prime}} \in S_h$.
\end{defn}

Before presenting the main results of this section, we first analyze the regularity of the function after applying the dimension reduction operator $\mathcal{R}$. Using an argument analogous to \cite[Lemma 3.4]{Bsource}, we can establish the following result. The detailed proof is omitted.

\begin{lem} \label{lem:RO}
	Let $\mathbf{g} \in H^m(S_h \times (-M,M))^3, m=1,2$. Then
	\begin{equation}\notag
			\mathcal{R}(\mathbf{g})(\mathbf{x^{\prime}}) \in H^m(S_h)^3.
	\end{equation}
Similarly, if $\mathbf{g} \in C^{1,\alpha}( \overline{S_h} \times (-M,M))^3, 0< \alpha <1 $, then
\begin{equation}\notag
	\mathcal{R}(\mathbf{g})(\mathbf{x^{\prime}}) \in C^{1,\alpha}( \overline{S_h} )^3.
\end{equation}
\end{lem}

In the following, we aim to demonstrate the vanishing property of transmission eigenfunctions at an edge corner in 3D. To begin, we introduce the mathematical setup. Let $S_h \Subset \mathbb{R}^2$ be defined by \eqref{eq:Sh}, and let $M>0$. Assume that  $x_3 \in (-M,M)$ and $L>0$, as defined in Definition \eqref{def:RO}. we assert that $L$ is sufficiently small for $(x_3-L, x_3+L) \Subset (-M,M)$. Consider $\mathbf{x}=(\mathbf{x^{\prime}}, x_3) \in \mathbb{R}^3$, where $\mathbf{x^{\prime}}\in \mathbb{R}^2$. Let $\mathbf{u}\in H^1(S_h \times (-M,M))^3$, $\mathbf{v} \in H^1(S_h \times (-M,M))$ and satisfy the following equation:

\begin{equation}\label{eq:transR3}
	\begin{cases}
		\mathcal{L}_{\lambda, \mu} \mathbf{u}+\omega^2\rho_e\mathbf{u}=\mathbf{0} \quad & \mathbf{x^{\prime}} \in S_h,  -M<x_3<M,\medskip\\
		\Delta v+\omega^2 \rho_b\kappa^{-1} v=0 \quad & \mathbf{x^{\prime}} \in S_h,  -M<x_3<M,\medskip\\
		\mathbf{u}\cdot\nu-\frac{1}{\rho_b\omega^2} \nabla v\cdot\nu=0 \quad &\mathbf{x^{\prime}} \in \Gamma_h^{\pm},  -M<x_3<M,\medskip\\
		T_{\nu}\mathbf{u}+v\nu=\mathbf{0} \quad & \mathbf{x^{\prime}} \in \Gamma_h^{\pm},  -M<x_3<M,
	\end{cases}
\end{equation}
where $\Gamma_h^{\pm}$ are defined in \eqref{eq:gammaEX}, $\nu$ is the outward normal vector to $\Gamma_h^{\pm} \times (-M,M)$, and $T_\nu$ is the boundary traction operator on $\Gamma_h^{\pm} \times (-M,M)$. For simplicity and clarity, we assume that $\rho_b$ and $\kappa$ are constants. Additionally, we assume that $\lambda, \mu$, and $\rho_e$ are constants, representing the properties of a homogeneous elastic background medium.

By applying the dimension reduction operator, as defined in Definition \ref{def:RO}, we can rewrite \eqref{eq:transR3} in an equivalent form, as shown in the following lemma.

\begin{lem}\label{lem:3uv}
	Suppose that  $\mathbf{u} \in H^1(S_h \times (-M,M))^3, v \in H^1(S_h \times (-M,M))$ fulfill \eqref{eq:transR3}. Denote 
	\begin{equation}\label{eq:3nu}
		\mathbf{u}= \begin{pmatrix}
			u_1\\ u_2\\ u_3
		\end{pmatrix}= \begin{pmatrix}
		\mathbf{u}^{(1,2)}\\ u_3
	\end{pmatrix} \in \mathbb{R}^3, \quad \nu = \begin{pmatrix}
	\nu_1\\ \nu_2\\ \nu_3
\end{pmatrix}= \begin{pmatrix}
\nu^{(1,2)}\\ \nu^{(3)}
\end{pmatrix} \in \mathbb{R}^3,
	\end{equation}
where $\nu$ is the outward normal vector to $\Gamma_h^{\pm} \times (-M,M)$. To ensure clarity and simplicity, we assume that $\nu^{(3)}=0$. Noting that $\nu^{(1,2)} \in \mathbb{R}^2$ signifies the exterior unit normal vector to $\Gamma_h^{\pm}$ and $\mathbf{u}^{(1,2)} \in \mathbb{R}^2$.
	 Denote
	\begin{align}
	\mathbf{G}_1(\mathbf{x^{\prime}})= & - \omega^2 \rho_e \mathcal{R}(\mathbf{u})(\mathbf{x^{\prime}}) - \int_{-L} ^L \psi^{\prime \prime} (x_3) \begin{pmatrix}
		\mu u_1\\ \mu u_2\\ (\lambda + 2 \mu) u_3
	\end{pmatrix}(\mathbf{x^{\prime}}, x_3) \rmd x_3 \notag\\
& + (\lambda + \mu) \int_{-L} ^L \psi^{\prime} (x_3) \begin{pmatrix}
	\partial_1 u_3\\ \partial_2 u_3\\ \partial_1 u_1 + \partial_2 u_2
\end{pmatrix}(\mathbf{x^{\prime}}, x_3) \rmd x_3, \notag\\
	\mathbf{G}_2(\mathbf{x^{\prime}})= & -\int_{-L}^L \psi^{\prime \prime}(x_3) v(\mathbf{x^{\prime}}, x_3) \rmd x_3- \omega^2 \rho_b \kappa^{-1} \mathcal{R}(v) (\mathbf{x^{\prime}}). \notag
	\end{align}
Then it holds that
\begin{equation}\label{eq:3UV}
\begin{cases}
	\tilde{\mathcal{L}} \mathcal{R}(\mathbf{u})(\mathbf{x^{\prime}}) = \mathbf{G}_1 (\mathbf{x^{\prime}})  & \mbox{in} \mbox{ } S_h, \medskip \\
	\Delta^{\prime}\mathcal{R}(v)(\mathbf{x^{\prime}})= \mathbf{G}_2 (\mathbf{x^{\prime}})  & \mbox{in} \mbox{ } S_h, \medskip \\
	\mathcal{R}(\mathbf{u})(\mathbf{x^{\prime}}) \cdot \nu- \frac{1}{\rho_b \omega^2} \partial_{\nu}\mathcal{R}(v)(\mathbf{x^{\prime}}) =0 & \mbox{on} \mbox{ } \Gamma_h^{\pm}, \medskip \\
	\begin{bmatrix}
		T_{\nu^{(1,2)}}\mathcal{R}(\mathbf{u}^{(1,2)}) + \lambda \mathcal{R}(\partial_3 u_3) \nu^{(1,2)}\\ 
		\mu \partial_{\nu^{(1,2)}} \mathcal{R}(u_3) + \mu \begin{bmatrix}
			\mathcal{R}(\partial_3 u_1)\\ \mathcal{R}(\partial_3 u_2) 
		\end{bmatrix}^\top \nu^{(1,2)}
	\end{bmatrix}
+\mathcal{R}(v) \nu ={\bf 0} & \mbox{on} \mbox{ } \Gamma_h^{\pm}, 
\end{cases}
\end{equation}
in the distribution sense, where $\Delta^{\prime} = \partial_1^2 + \partial_2^2 $ being Laplace operator with respect to $\mathbf{x^{\prime}}$-variables, $T_\nu$ is the two dimensional boundary traction operator defined in \eqref{eq:traction1} and 
\begin{equation}\label{eq:definedL}
 	\tilde{\mathcal{L}}= \begin{pmatrix}
 		\mu \Delta^{\prime} + (\lambda + \mu) \partial_1^2 & (\lambda + \mu) \partial_1 \partial_2 & 0\\ (\lambda + \mu) \partial_1 \partial_2 & 	\mu \Delta^{\prime} + (\lambda + \mu) \partial_2^2 & 0\\ 0 & 0 & \mu \Delta^{\prime}
 	\end{pmatrix} := \begin{pmatrix}
 	\mathcal{L} & 0\\ {\bf 0} & \mu \Delta^{\prime}
 \end{pmatrix},
\end{equation}
where $\mathcal{L}$ is the two dimensional Lam\'{e} operator with respect to $\mathbf{x^{\prime}}$-variables.
\end{lem}

\begin{rem}\label{rem:31}
	Under the same setup in Lemma \ref{lem:3uv}, the PDE system \eqref{eq:3UV} is equivalent to
	\begin{equation}\label{eq:3Du12}
		\begin{cases}
			\mathcal{L} \mathcal{R}(\mathbf{u}^{(1,2)})(\mathbf{x^{\prime}})= \mathbf{G}_1^{(1,2)} (\mathbf{x^{\prime}}) & \mbox{in} \mbox{ } S_h, \medskip \\
			T_{\nu^{(1,2)}}\mathcal{R}(\mathbf{u}^{(1,2)}) (\mathbf{x^{\prime}}) + \lambda \mathcal{R}(\partial_3 u_3) (\mathbf{x^{\prime}}) \nu^{(1,2)} + \mathcal{R}(v) (\mathbf{x^{\prime}}) \nu^{(1,2)} =\mathbf{0} & \mbox{on} \mbox{ } \Gamma_h^{\pm},
			\end{cases}
	\end{equation}
and
\begin{equation}\label{eq:3Du3}
	\begin{cases}
		\mu \Delta^{\prime} \mathcal{R}(u_3) (\mathbf{x^{\prime}}) = \mathbf{G}_1^{(3)} (\mathbf{x^{\prime}}) & \mbox{in} \mbox{ } S_h, \medskip \\
		\mu \partial_{\nu^{(1,2)}} \mathcal{R}(u_3) (\mathbf{x^{\prime}}) + \mu \begin{bmatrix}
			\mathcal{R}(\partial_3 u_1)(\mathbf{x^{\prime}})\\ \mathcal{R}(\partial_3 u_2) (\mathbf{x^{\prime}})
		\end{bmatrix} \cdot \nu^{(1,2)} =0 & \mbox{on} \mbox{ } \Gamma_h^{\pm},
	\end{cases}
\end{equation}
and
\begin{equation}\label{eq:3Dv}
	\begin{cases}
		\Delta^{\prime}\mathcal{R}(v)(\mathbf{x^{\prime}})= \mathbf{G}_2 (\mathbf{x^{\prime}})  & \mbox{in} \mbox{ } S_h, \medskip \\
		\mathcal{R}(\mathbf{u}^{(1,2)})(\mathbf{x^{\prime}}) \cdot \nu^{(1,2)}- \frac{1}{\rho_b \omega^2} \partial_{\nu^{(1,2)}}\mathcal{R}(v)(\mathbf{x^{\prime}}) = 0 & \mbox{on} \mbox{ } \Gamma_h^{\pm}, \medskip \\
	\end{cases}
\end{equation}
where
	\begin{align}
	\mathbf{G}_1^{(1,2)}(\mathbf{x^{\prime}})= & - \omega^2 \rho_e \mathcal{R}(\mathbf{u}^{(1,2)})(\mathbf{x^{\prime}}) - \int_{-L} ^L \psi^{\prime \prime} (x_3) \begin{pmatrix}
		\mu u_1\\ \mu u_2
	\end{pmatrix}(\mathbf{x^{\prime}}, x_3) \rmd x_3 \notag\\
	& + (\lambda + \mu) \int_{-L} ^L \psi^{\prime} (x_3) \begin{pmatrix}
		\partial_1 u_3\\ \partial_2 u_3
	\end{pmatrix}(\mathbf{x^{\prime}}, x_3) \rmd x_3, \notag\\
	\mathbf{G}_1^{(3)}(\mathbf{x^{\prime}})= & - \omega^2 \rho_e \mathcal{R}(\mathbf{u}_3)(\mathbf{x^{\prime}}) - (\lambda + 2 \mu)\int_{-L} ^L \psi^{\prime \prime} (x_3) u_3 (\mathbf{x^{\prime}}, x_3) \rmd x_3 \notag\\ & + (\lambda +  \mu)\int_{-L} ^L \psi^{\prime} (x_3) (\partial_1 u_1 + \partial_2 u_2 ) (\mathbf{x^{\prime}}, x_3)  \rmd x_3. \notag
\end{align}
	\end{rem}
	
We can present another main result of this paper on the geometrical structures  of transmission eigenfunctions at an edge corner in 3D.

\begin{thm}\label{thm:3Ru123}
	Let $\Omega \in \mathbb{R}^3$ be a bounded Lipschitz domain with $\mathbf{0} \in \partial \Omega$ and let $S_h \in \mathbb{R}^2$ be defined in \eqref{eq:Sh}. For any fixed $x_3 \in (-M,M), M>0$ and $L>0$ defined in Definition \ref{def:RO}, we suppose that $L$ is sufficiently small such that $(x_3-L,x_3+L) \Subset (-M,M)$ and $$(B_h \times (-M,M)) \cap \Omega=S_h \times(-M,M),$$ where $B_h \Subset \mathbb{R}^2$ is the central ball of radius $h \in \mathbb{R}_+$. Assume that  $(v,\mathbf{u})\in   H^1(\Omega) \times H^1(\Omega )^3$ be a pair of transmission eigenfunctions  \eqref{eq:trans2} satisfying \eqref{eq:transR3} and there exists a sufficiently small neighborhood $S_h \times (-M,M)$(i.e. $h>0$ is sufficiently small) of $(\mathbf{0},x_3)$ with $x_3 \in (-M,M)$ such that
	  $ \mathbf{u} \in C^{1,\alpha_1}(\overline S_h \times [-M,M])^3$ and $v \in C^{\alpha_2}(\overline S_h \times [-M,M])$  for $\alpha_1,\,\alpha_2 \in (0,1)$.  
Then one has
\begin{enumerate}
	\item[(a)] 
\begin{equation}\label{eq:thm31}
	\nabla^s \mathbf u (\mathbf 0)= \begin{pmatrix}
		\partial_1 u_1 (\mathbf 0)& 0 &0\\
		0& \partial_2 u_2 (\mathbf 0)& 0\\
		0& 0& \partial_3 u_3 (\mathbf 0)
		\end{pmatrix},
\end{equation}
	where $\nabla^s \mathbf u$ is the strain tensor of $\mathbf u$ defined in 
\eqref{eq:traction1}.
\item [(b)] Assume that $\partial_3 u_3(\mathbf{0}) = 0$,  we obtain that
\begin{align}\label{eq:thm31 v0}
	v(\mathbf 0)=-2(\lambda+\mu ) \partial_1 u_1(\mathbf 0 );
\end{align}
\item[(c)] Moreover, under the assumption $(b)$ and assume that $\partial_1 u_1(\mathbf{0})=0$, we deduce that
\begin{equation}\label{eq:3v0=0}
	v(\mathbf{0})=0;
	\end{equation}
\item[(d)] Furthermore, under the assumption  $(b)$ and $(c)$,  suppose that \begin{equation}\label{eq:3R1condition}
	\partial_2 u_1(\mathbf 0)=0,  \quad and \quad	\partial_3 u_\ell(\mathbf 0)=0,\quad \ell=1,2.
\end{equation}
Then
\begin{equation*}
	\nabla \mathbf{u}(\mathbf{0})=\mathbf{0}.
\end{equation*}
\end{enumerate} 
\end{thm}
\begin{rem}\label{Rem3-3}
Compared with the geometrical structures of the AE transmission functions $(v, \mathbf{u})$ for the 2D case in Theorem \ref{thm:thm21}, the strain $\nabla^s \mathbf{u}({\bf x}_c)$ at the edge corner point $\mathbf{x}_c$ is a diagonal matrix with diagonal components $\partial_j u_j({\bf x}_c)$ ($j=1, 2, 3$). In contrast, for the 2D case, $\nabla^s \mathbf{u}({\bf x}_c)$ is a scalar  matrix only depending on $\partial_1 u_1(\mathbf{x}_c)$. Furthermore,  to derive similar conclusions as for (b)-(d) in Theorem \ref{thm:thm21} for the 3D case, we need additional assumptions  $\partial_3 u_\ell(\mathbf x_c)=0,\  \ell=1,2,3. $ Due to the more complex geometrical configuration of the edge corner, the reduction operator is needed, which results in a more intricate integral identity around the edge corner. These additional assumptions in (b)-(d) on $\mathbf{u}$ are necessary to establish the geometrical structure for the AE transmission  functions $(v, \mathbf{u})$. 
\end{rem}

The proof of Theorem \ref{thm:3Ru123} is given in subsection \ref{subsec:32}. According to Remark \ref{rem:31}, we shall divide the proof of Theorem \ref{thm:3Ru123} into two separate parts.
In the first part, we deal with the system \eqref{eq:3Du12}, where we will perform a thorough analysis to prove the vanishing property of $\mathcal{R}(\mathbf{u}^{(1,2)})$ and $\mathcal{R}(\mathbf{v})$. In the second part, we shift our focus to the system \eqref{eq:3Du3} and perform a detailed analysis of its own. Before that, we need some necessary lemmas in the following subsection.

\vspace{2mm}

\subsection{Auxiliary lemmas of proving Theorem \ref{thm:3Ru123}} 

 In the following, Lemmas \ref{lem:3u12} and \ref{lem:u12}
 are derived by a rigorous analysis of the system \eqref{eq:3Du12}, which serve as fundamental ingredients for further deductions.


\begin{lem}\label{lem:3u12} 
	Let $S_h, \Lambda_h$ and $\Gamma_h^{\pm}$ be defined in \eqref{eq:Sh} and \eqref{eq:gammaEX}. Suppose that $\mathcal{R}(\mathbf{u}^{(1,2)}) \in H^1(S_h )^2$ and $\mathcal{R}(v) \in H^1(S_h)$  satisfy \eqref{eq:3Du12}, where $ \mathbf{u}^{(1,2)} \in H^1(S_h\times (-M,M))^2$ and $ v \in H^1(S_h\times (-M,M))$, Recall that CGO solutions $v_0 (\mathbf{x^{\prime}})$ and $ \mathbf{u}_0(\mathbf{x^{\prime}})$ are defined in \eqref{eq:cgo} and \eqref{eq:u0 cgo}, respectively. Denote
		\begin{equation}\label{eq:int inf2}
		\begin{split}
			I_{1,\star}^\pm &=\int_{\Gamma^\pm_h} v_0 (\mathbf{x^{\prime}}) \mathcal{R}(v) (\mathbf{x^{\prime}})\rmd \sigma,\quad I_{2,\star}^\pm = \int_{\Gamma_h^\pm  }v_0 (\mathbf{x^{\prime}}) \big[\mathcal{R}(\mathbf u^{(1,2)}) (\mathbf{x^{\prime}}) \cdot \mathbf p \big] \rmd \sigma,\\
			 I_{3,\star}^\pm & =\int_{\Gamma_h^\pm  }v_0 (\mathbf{x^{\prime}}) \mathcal{R}(\partial_3 u_3)(\mathbf{x^{\prime}}) \rmd \sigma, \quad
			I_{4,\star}=\int_{S_h}  v_0 (\mathbf{x^{\prime}}) \big[\mathbf{G}_1^{(1,2)} (\mathbf{x^{\prime}}) \cdot \mathbf{p}\big]\rmd \mathbf x^{\prime},\\ 
			 I_{\Lambda_h,2}^\star &=\int_{\Lambda_h}\big(\mathcal{R}(\mathbf u ^{(1,2)})(\mathbf{x^{\prime}}) \cdot T_\nu \mathbf u_0 (\mathbf{x^{\prime}}) -\mathbf u_0 (\mathbf{x^{\prime}}) \cdot T_\nu \mathcal{R}(\mathbf u^{(1,2)})(\mathbf{x^{\prime}})\big)\mathrm d \sigma.
		\end{split}
	\end{equation}
	Then we have
		\begin{align}
		0=& (\mathbf p \cdot \nu_M^{(1,2)}) I_{1,\star}^+ +(\mathbf p \cdot \nu_m^{(1,2)}) I_{1,\star}^- + 2 \mu \big( (\brho \cdot \nu_M^{(1,2)}) I_{2,\star}^+
		+ (\brho \cdot \nu_m^{(1,2)}) I_{2,\star}^-\big) \notag\\
		& + \lambda \big( (\mathbf p \cdot \nu_M^{(1,2)}) I_{3,\star}^+ + (\mathbf p \cdot \nu_m^{(1,2)}) I_{3,\star}^- \big) + I_{4,\star} +I_{\Lambda,2}^\star \label{eq:int u12}
	\end{align}
where $\brho, \mathbf{p} $ are defined in \eqref{eq:cgo} and \eqref{eq:p def}, respectively. Additionally,$\nu_m^{(1,2)} ,\nu_M^{(1,2)} \in \mathbb{R}^2$ are defined in \eqref{eq:nu pm}.
\end{lem}

\begin{proof}
	By the virtue of \eqref{eq:green1}, using \eqref{eq:u0 lame} and the boundary condition in \eqref{eq:3Du12}, one has
	\begin{align}\label{eq:int cgo u12}
		0=&\int_{S_h} \mathbf{u}_0(\mathbf{x^{\prime}}) \cdot \mathbf{G}_1^{(1,2)} (\mathbf{x^{\prime}}) \rmd \mathbf x^{\prime} + \int_{\Gamma_h^{\pm}} T_{\nu}\mathbf{u}_0 (\mathbf{x^{\prime}}) \cdot \mathcal{R}(\mathbf{u}^{(1,2)})(\mathbf{x^{\prime}}) \rmd \sigma \notag\\
		&+ \int_{\Gamma_h^{\pm}}\big[\lambda \mathcal{R}(\partial_3 u_3) (\mathbf{x^{\prime}}) \nu^{(1,2)} + \mathcal{R}(v) (\mathbf{x^{\prime}}) \nu^{(1,2)} \big] \cdot \mathbf{u}_0 (\mathbf{x^{\prime}})  \rmd \sigma + I_{\Lambda,2}^\star
	\end{align}
where $I_{\Lambda,2}^\star$ is defined in \eqref{eq:int inf2}.
Similar to \eqref{eq:Tu 238}, one has
\begin{equation}\label{eq:Tu1 238}
	\begin{split}
		\mathbf u_0 (\mathbf{x^{\prime}})\cdot \nu_M^{(1,2)}=(\mathbf p \cdot \nu_M^{(1,2)}) v_0 (\mathbf{x^{\prime}}) , \quad	T_\nu \mathbf u_0 (\mathbf{x^{\prime}}) |_{\Gamma_h^+}&=\mu [ ( \brho \cdot \nu_M^{(1,2)} )\mathbf p+(\mathbf p \cdot \nu_M^{(1,2)}) \brho ] v_0(\mathbf{x^{\prime}}),\\
		\mathbf u_0 (\mathbf{x^{\prime}})\cdot \nu_m^{(1,2)}=(\mathbf p \cdot \nu_m^{(1,2)}) v_0 (\mathbf{x^{\prime}}) ,\quad	 T_\nu \mathbf u_0 (\mathbf{x^{\prime}}) |_{\Gamma_h^-}&=\mu [ ( \brho \cdot \nu_m^{(1,2)} )\mathbf p+(\mathbf p \cdot \nu_m^{(1,2)}) \brho ] v_0(\mathbf{x^{\prime}}).
	\end{split}
\end{equation}
Besides, by the relationship of \eqref{eq:248 rho p}, we get
\begin{equation}\label{eq:rela p}
	(\mathbf{p} \cdot \nu^{(1,2)}) \int_{\Gamma_h^{\pm}} v_0 (\mathbf{x^{\prime}}) \big(\mathcal{R}(\mathbf{u}^{(1,2)})(\mathbf{x^{\prime}}) \cdot \brho\big) \rmd \sigma = (\brho \cdot \nu^{(1,2)}) \int_{\Gamma_h^{\pm}} v_0 (\mathbf{x^{\prime}}) \big(\mathcal{R}(\mathbf{u}^{(1,2)})(\mathbf{x^{\prime}}) \cdot \mathbf{p}\big) \rmd \sigma
	\end{equation}
Substituting \eqref{eq:Tu1 238},\eqref{eq:rela p} into \eqref{eq:int cgo u12}, after some algebraic calculations, we obtain \eqref{eq:int u12}.
\end{proof}
In the subsequent lemma, we thoroughly examine each integral term outlined in Lemma \ref{lem:3u12}.
\begin{lem}\label{lem:u12}
	Suppose that $(v,\mathbf{u}) \in  H^1(S_h \times (-M,M)) \times H^1(S_h \times (-M,M))^3 $ satisfying \eqref{eq:transR3}. Furthermore, assume that $ \mathbf{u} \in C^{1,\alpha_1}(\overline S_h \times [-M,M])^3$ and  $v \in C^{\alpha_2}(\overline S_h \times [-M,M])$ for $\alpha_1,\,\alpha_2 \in (0,1)$. Then the following integral estimations hold
	\begin{align}
		(\mathbf{p} \cdot \nu_M^{(1,2)})I_{1,\star}^+ & = -\frac{1}{s} \mathcal{R}(v)(\mathbf{0}) 
		  + \Oh(\frac{1}{s^{1+\alpha_2}}), \label{eq:I1+*}\\
		(\mathbf{p} \cdot \nu_m^{(1,2)}) I_{1,\star}^- & = \frac{1}{s} \mathcal{R}(v)(\mathbf{0}) 
		 + \Oh(\frac{1}{s^{1+\alpha_2}}), \label{eq:I1-*}		\\
		 (\brho \cdot \nu_M^{(1,2)})I_{2,\star}^+ & = - e^{- \bsi \phi} \big( e_1^\top \cdot \mathcal{R}(\mathbf{u}^{(1,2)})(\mathbf{0})\big) +  e_1^\top \nabla \mathcal{R}(\mathbf{u}^{(1,2)})(\mathbf{0}) \tau_M^{(1,2)} \frac{e^{-\bsi \theta_{M}}}{s}+ \Oh (\frac{1}{s^{1+\alpha_1}})\label{eq:I2+*},\\
		 (\brho \cdot \nu_m^{(1,2)})I_{2,\star}^+ & =  e^{- \bsi \phi} \big( e_1^\top \cdot \mathcal{R}(\mathbf{u}^{(1,2)})(\mathbf{0})\big) -  e_1^\top \nabla \mathcal{R}(\mathbf{u}^{(1,2)})(\mathbf{0}) \tau_m^{(1,2)} \frac{e^{-\bsi \theta_{m}}}{s}+ \Oh (\frac{1}{s^{1+\alpha_1}})\label{eq:I2-*},\\
		 	(\mathbf{p} \cdot \nu_M^{(1,2)})I_{3,\star}^+ & = -\frac{1}{s} \mathcal{R}(\partial_3 u_3)(\mathbf{0}) 
		 + \Oh(\frac{1}{s^{1+\alpha_1}}), \label{eq:I3+*}\\
		 (\mathbf{p} \cdot \nu_m^{(1,2)})I_{3,\star}^+ & = \frac{1}{s} \mathcal{R}(\partial_3 u_3)(\mathbf{0}) 
		 + \Oh(\frac{1}{s^{1+\alpha_1}}), \label{eq:I3-*}\
	\end{align}
as $s \rightarrow +\infty $, where $ \tau_m^{(1,2)} ,\tau_M^{(1,2)} ,\nu_m^{(1,2)} ,\nu_M^{(1,2)}\in \mathbb{R}^2$ be defined in\eqref{eq:taumM} and \eqref{eq:nu pm}, respectively.
\end{lem}
\begin{proof} 
	Since $ \mathbf{u} \in C^{1,\alpha_1}(\overline S_h \times [-M,M])^3$ and $\mathbf{u}$ is defined in \eqref{eq:3nu}, we have $\mathbf{u}^{(1,2)} \in C^{1,\alpha_1}(\overline S_h \times [-M,M])^2 $ and $u_3 \in C^{1,\alpha_1}(\overline S_h \times [-M,M])$. By Lemma \ref{lem:RO}, it follows that $\mathcal{R}(\mathbf{u}^{(1,2)})(\mathbf{x^{\prime}}) \in C^{1,\alpha_1}(\overline S_h)^2$, $\mathcal{R}(\mathbf{u}_3)(\mathbf{x^{\prime}}) \in C^{1,\alpha_1}(\overline S_h )$ and $\mathcal{R}(\partial_3 u_\ell)(\mathbf{x^{\prime}}) \in C^{\alpha_1}(\overline S_h), \ell=1,2,3$. Therefore, we have
	\begin{align}
	&\mathcal{R}(\mathbf{u}^{(1,2)})(\mathbf{x^{\prime}}) =	\mathcal{R}(\mathbf{u}^{(1,2)})(\mathbf{0}) + r \nabla \mathcal{R}(\mathbf{u}^{(1,2)})(\mathbf{0}) \boldsymbol{\tau } + \delta \mathcal{R}(\mathbf{u}^{(1,2)})(\mathbf{x^{\prime}}), \label{eq:Ru0}\\
	&|\delta \mathcal{R}(\mathbf{u}^{(1,2)})(\mathbf{x^{\prime}})| \leq \|  \mathcal{R}(\mathbf{u}^{(1,2)})(\mathbf{x^{\prime}}) \|_{C^{1,\alpha _1} (\overline{S_h})^2 } |\mathbf x^{\prime}|^{1+\alpha_1},\notag\\
	&\mathcal{R}(u_3)(\mathbf{x^{\prime}})=\mathcal{R}(u_3)(\mathbf 0) + r \nabla \mathcal{R}(u_3)(\mathbf{0}) \boldsymbol{\tau } + \delta \mathcal{R}(u_3)(\mathbf{x^{\prime}}),\label{eq:Ru30}\\
	&|\delta \mathcal{R}(u_3)(\mathbf{x^{\prime}})| \leq \|  \mathcal{R}(u_3)(\mathbf{x^{\prime}}) \|_{C^{1,\alpha_1 } (\overline{S_h} )} |\mathbf x^{\prime}|^{1+\alpha_1},\notag\\
	&\mathcal{R}(\partial_3 u_\ell)(\mathbf{x^{\prime}})=\mathcal{R}(\partial_3 u_\ell)(\mathbf 0) + \delta \mathcal{R}(\partial_3 u_\ell)(\mathbf{x^{\prime}}),\quad \ell=1,2,3,\label{eq:Ru330}\\
	&|\delta \mathcal{R}(\partial_3 u_\ell)(\mathbf{x^{\prime}})| \leq \|  \mathcal{R}(\partial_3 u_\ell)(\mathbf{x^{\prime}}) \|_{C^{\alpha_1 } (\overline{S_h})} |\mathbf x^{\prime}|^{\alpha_1},\notag
	\end{align}
 where $\mathbf x^{\prime}=r\boldsymbol{\tau }\in \mathbb R^2$ with $\boldsymbol{\tau }=(\cos\theta, \sin \theta )$ and $r\in \mathbb R_+\cup \{0\}$. The subsequent argument is similar to Lemma \ref{lem:estimate}, here we omitted that.
\end{proof}
In the subsequent discussion, our primary focus lies on  \eqref{eq:3Du3}. We shall get some estimates for the integral equations and obtain a crucial formula that plays a significant role in proving Theorem \ref{thm:3Ru123}.
 

\begin{lem}\label{lem:35u3}
Recall that CGO solution $v_0 (\mathbf{x^{\prime}})$ is defined in \eqref{eq:cgo}.  Denote
			\begin{align}
				&\hat I_{1}^\pm =\int_{\Gamma^\pm_h} v_0 (\mathbf{x^{\prime}}) \mathcal{R}(u_3) (\mathbf{x^{\prime}})\rmd \sigma,  
				, \notag\\
				&\hat I_{2}^\pm  =\int_{\Gamma_h^\pm  }v_0 (\mathbf{x^{\prime}})
				\begin{bmatrix}
					\mathcal{R}(\partial_3 u_1)(\mathbf{x^{\prime}})\\ \mathcal{R}(\partial_3 u_2)(\mathbf{x^{\prime}})
				\end{bmatrix} \cdot \nu_M^{(1,2)}  \rmd \sigma, 
				\quad \hat I_{3}=\int_{S_h}  v_0 (\mathbf{x^{\prime}}) \mathbf{G}_1^{(3)} (\mathbf{x^{\prime}}) \rmd \mathbf x^{\prime},\notag\\ 
				&\hat I_{\Lambda_h,2} =\int_{\Lambda_h}\big(\mathcal{R}( u _3)(\mathbf{x^{\prime}}) \partial_{\nu} v_0 (\mathbf{x^{\prime}}) - v_0 (\mathbf{x^{\prime}})
				\partial_{\nu} \mathcal{R}(u_3)  (\mathbf{x^{\prime}})\mathrm d \sigma, \label{eq:u3v}
			\end{align}
where $\mathbf{u}$  and $\nu$ be defined in \eqref{eq:3nu}. Under the same assumption in Lemma \ref{lem:u12}, the following integral estimations hold:		
\begin{align}
&(	\mathbf{\brho} \cdot \nu_M^{(1,2)} ) \hat I_1^+= -\bsi \mathcal{R}(u_3)(\mathbf{0}) +  \frac{\bsi e^{\bsi(\phi-\theta_M)}}{s} \nabla \mathcal{R}(u_3)(\mathbf{0})  \cdot \tau_M^{(1,2)} + \Oh(\frac{1}{s^{1+\alpha_1}}),\label{eq:I1+h}\\
&(	\mathbf{\brho} \cdot \nu_m^{(1,2)} ) \hat I_1^-= \bsi \mathcal{R}(u_3)(\mathbf{0}) - \frac{\bsi e^{\bsi(\phi-\theta_M)}}{s} \nabla \mathcal{R}(u_3)(\mathbf{0})  \cdot \tau_m^{(1,2)} + \Oh(\frac{1}{s^{1+\alpha_1}}),\label{eq:I1-h}\\
& \hat I_2^+ = - \frac{e^{\bsi (\phi-\theta_{M})}}{s} \begin{bmatrix}
	\mathcal{R}(\partial_3 u_1)(\mathbf{0})\\ \mathcal{R}(\partial_3 u_2)(\mathbf{0})
\end{bmatrix} \cdot \nu_M^{(1,2)} +\Oh(\frac{1}{s^{1+\alpha_1}}),\label{eq:I3+h}\\
& \hat I_2^- = - \frac{e^{\bsi (\phi-\theta_{m})}}{s} \begin{bmatrix}
	\mathcal{R}(\partial_3 u_1)(\mathbf{0})\\ \mathcal{R}(\partial_3 u_2)(\mathbf{0})
\end{bmatrix}  \cdot \nu_m^{(1,2)} +\Oh(\frac{1}{s^{1+\alpha_1}}),\label{eq:I3-h}
\end{align}		
as $s\rightarrow + \infty$.
Furthermore, we can observe that
\begin{align}\label{eq:gra3u3}
	\bsi  \nabla\mathcal{R}(u_3)(\mathbf{0}) \cdot \big(\tau _M^{(1,2)} e^{-\bsi \theta_M} -  \tau _m^{(1,2)} e^{-\bsi \theta_m }\big)
	 - \begin{bmatrix}
		\mathcal{R}(\partial_3 u_1)(\mathbf{0})\\ \mathcal{R}(\partial_3 u_2)(\mathbf{0})
	\end{bmatrix}  \cdot \bigg(  \nu_M^{(1,2)} e^{-\bsi \theta_M}+  \nu_m^{(1,2)} e^{-\bsi \theta_m} \bigg)=0.
	\end{align}
		\end{lem}
\begin{proof}
 
From \eqref{eq:GIN1} and \eqref{eq:v0 delta} along with the boundary condition in \eqref{eq:3Du3}, one has
	\begin{equation}\label{eq:3Du31}
		(\brho \cdot \nu_M^{(1,2)}) \hat I_1^+ + 	(\brho \cdot \nu_m^{(1,2)}) \hat I_1^- + \hat I_2^+ + \hat I_2^- + \mu^{-1} \hat I_3 + \hat I_{\Lambda_h,2} = 0,
	\end{equation}
	then substituting \eqref{eq:Ru30}  into the expression of $\hat I_1^{\pm}$ defined in \eqref{eq:u3v}, and using  \eqref{eq:rho 220a}, it arrives at
	\begin{align}
		(\brho \cdot \nu_M^{(1,2)}) \hat I_1^+= \bsi s e^{\bsi (\theta_M- \phi)} \int_{\Gamma_h^+} v_0(\mathbf{x^\prime})\big[\mathcal{R}(\mathbf{u}_3)(\mathbf 0) + r \nabla \mathcal{R}(\mathbf{u}_3)(\mathbf{0}) \tau_M^{(1,2)}\big] \rmd \sigma + r_{\hat I_1^+}\label{eq:I1h ex},\\
		(\brho \cdot \nu_m^{(1,2)}) \hat I_1^-= - \bsi s e^{\bsi( \theta_m- \phi)} \int_{\Gamma_h^+} v_0(\mathbf{x^\prime})\big[\mathcal{R}(\mathbf{u}_3)(\mathbf 0) + r \nabla \mathcal{R}(\mathbf{u}_3)(\mathbf{0}) \tau_m^{(1,2)}\big] \rmd \sigma + r_{\hat I_1^+},\notag
	\end{align}
 where
\begin{equation}\notag
	r_{\hat I_1^{\pm}}=\int_{\Gamma_h^{\pm}}v_0(\mathbf{x^{\prime}}) (\delta \mathcal{R}(u_3)) \rmd \sigma.
\end{equation}
By Cauchy-Schwarz inequality, \eqref{eq:v0 gamma al},  and \eqref{eq:Ru30}, we obtain that
\begin{align}\label{eq:I1hr}
	\left|r_{\hat I_1^\pm }\right| &\leq   \|\mathcal{R} (u_3)\|_{C^{1,\alpha_1} (\overline{S_h} \times (-M,M) )   } \int_{\Gamma_h^\pm  } | v_0(\mathbf x^{\prime})| |\mathbf x^{\prime}|^{1+\alpha_1 }\rmd \sigma\leq \frac{ \Gamma(\alpha_1+2) }{ (\delta s)^{\alpha_1+2}} + \Oh(e^{-s\delta h/2} ),
\end{align}
as $s\rightarrow + \infty. $ Substituting \eqref{eq:I2 253 est} and \eqref{eq:I1hr} into \eqref{eq:I1h ex}, we can obtain \eqref{eq:I1+h} and \eqref{eq:I1-h}.

 In addition to that,\eqref{eq:I3+h} and \eqref{eq:I3-h} can be subjected to a similar framework, which is omitted. According to the expression of $\hat I_{\Lambda_h,2}$ be defined in \eqref{eq:u3v}, and using Lemma \ref{lem:I}, it is clear that $I_{\Lambda_h,2}^{\star}$ decays exponentially as $s \rightarrow + \infty$. Besides, from the expression of $\hat I_3$ defined in \eqref{eq:u3v}, we obtain
\begin{equation}\label{eq:3Su12}
	\hat{I}_{3}= \Oh(\frac{1}{s^2}) \quad \mbox{as} \quad  s \rightarrow + \infty.
\end{equation}
Therefore, substituting \eqref{eq:I1+h}--\eqref{eq:I3-h} and \eqref{eq:3Su12} into  \eqref{eq:3Du31}, after that multiplying $s/e^{\bsi \phi}$ on both sides, then let $s\rightarrow + \infty$, we can obtain \eqref{eq:gra3u3}.
\end{proof}

After establishing Lemma \ref{lem:3u12}  to Lemma \ref{lem:35u3}, we are now ready to proceed with the proof of the Theorem \ref{thm:3Ru123}.

\subsection{The proof of Theorem \ref{thm:3Ru123}}\label{subsec:32}
\begin{proof}
According to the expression of $I_{\Lambda_h,2}^{\star}$ be defined in \eqref{eq:int inf2}, and using Lemma \ref{lem:I}, we can easily see that $I_{\Lambda_h,2}^{\star}$ decays exponentially as $s \rightarrow + \infty$. Besides, From the expression of $I_{4,\star}$ be defined in \eqref{eq:int inf2}, we can obtain that
	\begin{equation}\label{eq:Su12}
		I_{4,\star}= \Oh(\frac{1}{s^2}).
	\end{equation}
	Therefore, substituting \eqref{eq:I1+*}--\eqref{eq:I3-*} and \eqref{eq:Su12} into  \eqref{eq:int u12}, after that multiplying $s$ on both sides, we can get 
	\begin{equation}\label{eq:Ru0}
		e_1^\top \cdot (\nabla \mathcal{R}(\mathbf{u}^{(1,2)})(\mathbf{0}) \tau_M^{(1,2)}) e^{-\bsi \theta_{M}} -  e_1^\top \cdot(\nabla \mathcal{R}(\mathbf{u}^{(1,2)})(\mathbf{0}) \tau_m^{(1,2)}) e^{-\bsi \theta_{m}}=0,
	\end{equation}
	where $e_1=(1,\bsi)^\top$ and $\tau_m^{(1,2)} ,\tau_M^{(1,2)} \in \mathbb{R}^2$ be defined in\eqref{eq:taumM}. Then complying with Lemma \ref{lem:conclusion}, we can obtain
	\begin{equation}\label{eq:Ru01}
		\mathcal{R}(\partial_1 u_1) (\mathbf 0)= \mathcal{R}(\partial_2 u_2) (\mathbf 0) \mbox{ and } \mathcal{R}(\partial_2 u_1) (\mathbf 0)+\mathcal{R}(\partial_1 u_2) (\mathbf 0) =0.
	\end{equation}
By utilizing equation \eqref{eq:gra3u3}
we know that
\begin{align}
	\nabla\mathcal{R}(u_3)(\mathbf{0})^\top \cdot \tau _M^{(1,2)} e^{-\bsi \theta_M}  = & \int_{-L}^L \psi(x_3) \begin{bmatrix}
		\partial_1 u_3 & \partial_2 u_3
	\end{bmatrix}(\mathbf{x^{\prime}}, x_3) \rmd x_3 
	\begin{bmatrix}
		\cos \theta_M \\ \cos \theta_m
	\end{bmatrix} e^{-\bsi \theta_M}\notag \\
	= & \mathcal{R}(\partial_1 u_3) (\mathbf{0}) \cos^2 \theta_M + \mathcal{R}(\partial_2 u_3) (\mathbf{0}) \sin \theta_M \cos \theta_M\label{eq:3R12u3M}\\
	&- \bsi \big( \mathcal{R}(\partial_1 u_3) (\mathbf{0}) \sin \theta_M \cos \theta_M +  \mathcal{R}(\partial_2 u_3) (\mathbf{0}) \sin^2 \theta_M \big).\notag
\end{align}
Using a similar way, we obtain that
\begin{align}
	\nabla\mathcal{R}(u_3)(\mathbf{0})^\top \cdot \tau _m^{(1,2)} e^{-\bsi \theta_m}  
	= & \mathcal{R}(\partial_1 u_3) (\mathbf{0}) \cos^2 \theta_m + \mathcal{R}(\partial_2 u_3) (\mathbf{0}) \sin \theta_m \cos \theta_m\label{eq:3R12u3m}\\
	&- \bsi \big( \mathcal{R}(\partial_1 u_3) (\mathbf{0}) \sin \theta_m \cos \theta_m +  \mathcal{R}(\partial_2 u_3) (\mathbf{0}) \sin^2 \theta_m \big).\notag
\end{align}
Furthermore, by separating the real and imaginary parts, after some straightforward calculations, we can derive the following relational expressions:
	\begin{equation}\label{eq:Ru123}
		\mathcal{R}(\partial_1 u_3) (\mathbf 0) + \mathcal{R}(\partial_3 u_1) (\mathbf 0) =0 \mbox{ and } \mathcal{R}(\partial_2 u_3) (\mathbf 0)+\mathcal{R}(\partial_3 u_2) (\mathbf 0) =0.
		\end{equation}
By combining \eqref{eq:Ru01} with \eqref{eq:Ru123}, we can deduce that \eqref{eq:thm31} holds by the definition of reduction operator $\mathcal{R}$ in \eqref{eq:RO}.

	Recall that $\nu_m^{(1,2)}$ is the outward unit normal vector to $\Gamma_h^-$, as defined in \eqref{eq:nu pm}. For convenience, let $\theta_m=0$. Substituting \eqref{eq:nu pm} into the boundary condition given in \eqref{eq:3Du12} and taking $\mathbf{x^{\prime}} = \mathbf{0}$, we can deduced that 
	\begin{equation}\label{eq:BC3}
		\lambda \mathcal{R}(\partial_1 u_1)(\mathbf{0}) + (\lambda+2\mu) \mathcal{R}(\partial_2 u_2) (\mathbf{0}) + \lambda\mathcal{R}(\partial_3 u_3) (\mathbf{0}) + \mathcal{R}(v)(\mathbf{0})=0
	\end{equation}
	Basing on the assumption $\partial_3 u_3(\mathbf{0})=0$ and using \eqref{eq:Ru01}, we readily obtain \eqref{eq:thm31 v0}.  Furthermore, using the assumption that $\partial_1 u_1(\mathbf{0})=0$ and \eqref{eq:thm31 v0}, we have \eqref{eq:3v0=0}. Thus, combining the equations \eqref{eq:thm31} and \eqref{eq:3v0=0} with the condition in \eqref{eq:3R1condition}, we obtain 
	\begin{equation}\notag
		\nabla \mathbf{u}(\mathbf{0})=\mathbf{0}.
	\end{equation}
 The proof is completed.
\end{proof}

\section{Visbility and unique recovery results for the inverse scattering problem}\label{sec:ip}

In this section, we consider the scattering problem described by \eqref{eq:transea1} in $\mathbb R^2$, which models the corresponding scattering phenomena between an acoustic medium embedded in a homogeneous elastic background and an elastic incident wave fulfilling \eqref{eq:ei1}. Our focus is on the visibility and unique identifiability of the acoustic medium. First, we demonstrate that an acoustic medium scatterer with a convex planar corner must scatter for any incident compressional or shear elastic wave. Following this, we show that a convex polygonal acoustic scatterer within a homogeneous elastic medium can be uniquely determined by at most three far-field measurements in generic scenarios. The key to our analysis is the vanishing property of AE transmission eigenfunctions near a convex planar corner, which plays a crucial role.



Let $\Omega$ be a bounded Lipschitz domain in $\mathbb{R}^2$ with a connected complement $\mathbb{R}^2 \backslash \overline\Omega$. In what follows, let $\mathbf{d} \in \mathbb{S}^1$ denote the incident direction, $\mathbf{d}^\perp \in \mathbb{S}^1$ be orthogonal to $\mathbf{d} $. And $k_p, k_s$ are compressional and shear wave numbers defined in \eqref{eq:kpks}. Consider an incident elastic plane wave field $\mathbf{u}^{in}$ given by
\begin{equation}\label{eq:planekpks}
	\mathbf{u}^i:=\mathbf{u}^i(\mathbf{x};k_p,k_s,\mathbf{d}) = \alpha_p \mathbf{d} e^{\bsi k_p \mathbf{x}\cdot \mathbf{d}} + \alpha_s \mathbf{d}^\perp e^{\bsi k_s \mathbf{x}\cdot \mathbf{d}} ,\quad \alpha_p,\alpha_s \in \mathbb{C}, \quad |\alpha_p| + |\alpha_s| \neq 0.
\end{equation}
In the context of wave propagation, the acoustic medium denoted by $\Omega\Subset \mathbb{R}^2$ with the density $\rho_b$ and modulus $\kappa$ plays a key role by interrupting the propagation of the elastic wave, resulting in the emergence of a scattered elastic wave field, denoted as $\mathbf{u}^{sc}$. The scattering problem is modeled in \eqref{eq:transea1}. Since in this paper, we mainly focus on the inverse problem \eqref{eq:so1} corresponding to \eqref{eq:transea1}, we assume that there exists a unique solution $v\in H^1(\Omega)$ and $\mathbf u \in H_{\rm loc}^1(\mathbb R^2)^2$ to \eqref{eq:transea1}. The inverse problem \eqref{eq:so1} concerns the determination of the  scatterer $\Omega$ by knowledge of the far-field pattern $\mathbf{u}_t^{\infty}(\hat{\mathbf{x}};\mathbf{d})$, where $\mathbf{u}_t^{\infty}$ is given in \eqref{eq:far-field}.

Considering the scattering problem \eqref{eq:transea1}, an acoustic medium scatterer is considered invisible if the far-field pattern \eqref{eq:far-field} of the elastic scattered wave, as defined in  \eqref{eq:transea1}, is zero. Otherwise, it is visible from the corresponding far-field measurements. In Theorem \ref{thm:radiating}, we shall reveal that when the underlying acoustic medium scatterer possesses a convex planar corner, it's always visible. Before that, considering an interface problem \eqref{eq:inverse12}  between an acoustic and an elastic wave, applying \cite[Theorem 2.1]{ZM},  we can establish the H\"older continuous regularity of the acoustic wave up to the interface, especially at the corner point.

\begin{lem}\label{lem:regu1}
Let $S_{h} = W \cap B_{h}$ and $\Gamma_{h}^{\pm} = \partial S_{h}\backslash \partial B_{h} $, where $W$ is the infinite sector defined in \eqref{eq:W} with the opening angle $\theta_W\in (0,\pi)$. Suppose that $v\in H^1(B_{h})$ and  $\mathbf{u}$ is real analytic in $B_h$, which satisfy
\begin{equation}\label{eq:inverse12}
	\begin{cases}
		\Delta v +\omega^2 \rho_b\kappa^{-1} v =0\quad & \mbox{in}\ S_{h},\medskip\\
		\mathcal{L}_{\lambda, \mu} \mathbf{u} +\omega^2\rho_e\mathbf{u}={\bf 0}\quad & \mbox{in}\ B_{h},\medskip\\
		\mathbf{u}\cdot\nu-\frac{1}{\rho_b\omega^2} \nabla v\cdot\nu=0\quad & \mbox{on}\ \Gamma_{h}^{\pm},\medskip\\
		T_{\nu}\mathbf{u} + v \nu={\bf 0}\quad & \mbox{on}\ \Gamma_{h}^{\pm},\medskip\\
		\end{cases}
	\end{equation}
where $\omega, \rho_b$ and $\kappa$ be positive real-valued constants.  Then there exists $\alpha \in (0,1)$ such that $v \in C^{\alpha}(\overline {S_{h/2}})$.
\end{lem}

\begin{proof}
	According to the geometric setup of $W$, we  know that $\mathbf y_1=h/2(\cos \theta_m, \sin \theta_m )$ and $\mathbf y_2=h/2(\cos \theta_M, \sin \theta_M )$ are two endpoints of line segments $\Gamma_{h/2}^-$ and $\Gamma_{h/2}^+$, respectively, where $\Gamma_{h/2}^\pm=B_{h/2}\cap W$. Denote $\Lambda=\mathbf y_2-\mathbf y_1$. Let the triangle $D$ be formed by three line segments $\Gamma_h^-$, $\Gamma_h^+$ and $\Lambda$. It is easy to see that $D\Subset S_h$. By interior regularity of elliptic PDE, we know that $v$ is analytic on $\Lambda$. Due to \eqref{eq:inverse12}, it yields that
	\begin{equation*}
		\begin{cases}
			\Delta v = f  \quad & \mbox{in } D, \medskip\\
			v=h \quad & \mbox{on } \Lambda , \medskip\\
				\frac{\partial v}{\partial \nu}=g_+ \quad & \mbox{on } \Gamma^+_{h/2},  \medskip\\
				\frac{\partial v}{\partial \nu}=g_- \quad & \mbox{on } \Gamma^-_{h/2},
		\end{cases}
	\end{equation*}
	where $f=- \omega^2 \rho_b \kappa^{-1} v\in L^2 (D)$, $g_+= \rho_b \omega^2 (\mathbf u\cdot \nu)|_{\Gamma_{h/2}^+} \in H^{1/2} (\Gamma_{h/2}^+ )  $, $g_-= \rho_b \omega^2 (\mathbf u\cdot \nu)|_{\Gamma_{h/2}^-} \in H^{1/2} (\Gamma_{h/2}^-)  $ and $h=v|_{\Lambda} \in H^{3/2} (\Lambda)$. By noting $\theta_W\in (0,\pi)$, using \cite[Theorem 2.1]{ZM}, we know that $v \in H^2(S_{h/2})$. By Sobolev embedding, we prove this lemma. 
\end{proof}

\begin{thm}\label{thm:radiating}
	Consider the scattering problem \eqref{eq:transea1}. Let $\Omega$ be the acoustic medium scatterer with configurations $\rho_b$ and $\kappa$ associated with \eqref{eq:transea1}. The homogenous isotropic elastic medium background  $\mathbb R^2$ is characterized by medium parameters $\lambda$, $\mu$, and $\rho_e$. Suppose that the elastic incident wave associated with \eqref{eq:transea1} given by the elastic compressional incident wave $\mathbf{u}^i_{\rm p}$, where 
	\begin{equation}\label{eq:incident}
		\mathbf{u}^i_{\rm p}=\mathbf{d} e^{\bsi k_p \mathbf{x}\cdot \mathbf{d}} .
	\end{equation}
	If $\Omega$ has a convex planar corner, then $\Omega$ always scatters for any elastic incident wave $\mathbf{u}^i_{\rm p}$. 
	
	\end{thm} 
	
	\begin{proof}
		We prove this theorem by contradiction. Suppose that  $\Omega$ has a convex planar corner point $\mathbf x_c \in \partial \Omega$  and does not scatter. Then $\Omega$ is invisible. Suppose that the considering corner is formed by two line segments  $\mathbf y_+-\mathbf x_c$	 and $\mathbf y_--\mathbf x_c$ with $|\mathbf y_+-\mathbf x_c|=|\mathbf y_--\mathbf x_c|$, where $\mathbf y_+ \in  \partial\Omega$ and $\mathbf y_-\in \partial \Omega$. 		
		Since the partial differential operators $\Delta$ and $\mathcal L_{\lambda,\mu}$ are invariant under rigid motions, without loss of generality, we assume that $\mathbf x_c=\mathbf 0$ and the corner is described by \eqref{eq:Sh}. For simplicity,  we assume that the underlying rigid motion is a rotation, which is characterized by an orthogonal matrix $Q$. Let $(v,\mathbf u)\in H^1(\Omega) \times H_{\rm loc}^1(\mathbb R^2)^2$ be the solution to \eqref{eq:transea1} associated with $\Omega$ and $\mathbf{u}^i_{\rm p}$.  By the Rellich Theorem (cf. \cite{Hahner98}) and unique continuation principle, one has
		\begin{equation} \label{eq:trans invisible}
\begin{cases}
\mathcal{L}_{\lambda, \mu} \mathbf{u}^i_{\rm p}+\omega^2\rho_e\mathbf{u}^i_{\rm p}={\bf 0}\quad & \mbox{in}\ S_h,\medskip\\
\Delta v+\omega^2 \rho_b\kappa^{-1} v=0\quad & \mbox{in}\ S_h,\medskip\\
\mathbf{u}^i_{\rm p}\cdot\nu-\frac{1}{\rho_b\omega^2} \nabla v\cdot\nu=0\quad & \mbox{on}\ \Gamma_h^\pm,\medskip\\
T_{\nu}\mathbf{u}^i_{\rm p}+v\nu={\bf 0}\quad & \mbox{on}\ \Gamma_h^\pm.
\end{cases}
\end{equation}

Using Lemma \ref{lem:regu1} and Theorem \ref{thm:thm21}, we know that 
\begin{align}\label{eq:thm41 eq}
	\nabla^s \mathbf u^i_{\rm p} (\mathbf 0)=\partial_1 u_{\rm p,1}^i (\mathbf 0)	I,
\end{align}
where $\mathbf u^i_{\rm p} =(\mathbf u^i_{\rm p,1} ,\mathbf u^i_{\rm p,2} )$. 


Let $\widetilde {\mathbf d}=Q\mathbf d =(d_1,d_2)$. Using \eqref{eq:incident} and  \eqref{eq:thm41 eq}, by directly calculations, it yields that $d_1^2=d_2^2$ and $d_1d_2=0$. This contradicts to the fact that $ \widetilde {\mathbf d} \in \mathbb S^1$. 
	\end{proof}
	
	\begin{rem}
		Theorem \ref{thm:radiating} generically holds for any shear incident wave $\mathbf{u}^i_{\rm s}=\mathbf{d}^\perp e^{\bsi k_s \mathbf{x}\cdot \mathbf{d}}$. Namely, when the underlying scatterer $\Omega$ has a convex corner, it radiates a nonzero far field pattern if one uses any shear incident wave  $\mathbf{u}^i_{\rm s}$. Using a similar argument, by contradiction, we can get a similar transmission eigenvalue problem \eqref{eq:trans invisible}. By Theorem \ref{thm:thm21}, from \eqref{eq:thm21 v0} we can show that 
		\begin{align}\label{eq:generic}
				v(\mathbf x_c)+2(\lambda+\mu ) \partial_1 u^i_{\rm s, 1}(\mathbf x_c )=0,
		\end{align}
		where $\mathbf x_c$ is the corner point and $\mathbf u^i_{\rm s} =(\mathbf u^i_{\rm s,1} ,\mathbf u^i_{\rm s,2} )$. Here $v$ is the acoustic total wave field to \eqref{eq:transea1} associated with $\Omega$ and $\mathbf u^i_{\rm s}$. The equation \eqref{eq:generic} generically does not hold from a physical point of view. Hence a scatterer possessing a convex corner that always scatters for any shear incident wave. 
	\end{rem}

 In Theorem \ref{thm:radiating}, we have shown that if the scatterer has a convex corner, then it is visible since the corresponding far-field pattern cannot be zero. In Theorem \ref{thm:inverse1}, we shall reveal that when visibility holds for the scatterer, the shape unique determination by at most three far-field measurements can be achieved. We first consider the local uniqueness result. When the scatterer is a convex polygon, we next prove a global uniqueness.  First, we introduce the admissible class of the underlying acoustic medium scatterer in our study. 
\begin{defn}\label{def:admissible}
	Let $\Omega \Subset \mathbb{R}^2$ be a bounded, simply connected Lipschitz domain in $\mathbb{R}^2$. Considering the scattering problem \eqref{eq:transea1} associated with $\Omega$, let the corresponding  elastic total wave field  be $\mathbf u=(u_1,u_2)$. Then the scatterer  $\Omega$ is said to be admissible if  it satisfies either of the following two conditions:
	\begin{enumerate}
	
	\item[(a)]  The strain $\nabla^s \mathbf u (\mathbf x)$ is not a scalar matrix in the sense that 
for any $\mathbf{x}\in \mathbb{R}^2\setminus \overline \Omega $,
		\begin{equation}\label{eq:admissible strain}
		\nabla^s \mathbf u (\mathbf{x})\neq \partial_1 u_1(\mathbf x) I.
		\end{equation}
	
		\item[(b)] The gradient of the total elastic wave field $\mathbf u$ is non-vanishing everywhere in the sense that for any $\mathbf{x}\in \mathbb{R}^2\setminus \overline \Omega$,
		\begin{equation}\label{eq:admissible}
		\nabla \mathbf u (\mathbf{x})\neq \mathbf 0.
		\end{equation}
\end{enumerate}	
\end{defn}

\begin{rem}
The admissible assumptions \eqref{eq:admissible strain} and \eqref{eq:admissible} can be satisfied in generic physical scenarios. Specifically, \eqref{eq:admissible strain} is generally valid based on physical intuition. If \eqref{eq:admissible strain} is violated, then the strain  $\nabla^s \mathbf u (\mathbf x)$ would reduce to a scalar matrix, depending solely on $\partial_1 u_1(\mathbf x)$, which does not typically occur in most physical situations.

Moreover, the fulfillment of assumption \eqref{eq:admissible}  can be illustrated through a specific example. Consider a scenario where the angular frequency $\omega \in \mathbb{R}_+$  is sufficiently small. Physically, this corresponds to a case where the size of the scatterer, represented by diam($\Omega$), is small compared to the wavelength of the incident wave. In such a scenario, the scatterer's disturbance to the incident wave should be negligible, implying that the scattered wave  $\mathbf{u}^{s}$  is small in comparison to the incident wave $\mathbf{u}^{i}$. Thus, if 
$\mathbf{u}^{i}$  is non-vanishing everywhere (for example, when 
$\mathbf{u}^{i}$ is a elastic compressional incident wave), then the total wave  $\mathbf{u}=\mathbf{u}^{i} + \mathbf{u}^{s}$
  should also be non-vanishing throughout the domain.  We shall not delve into this aspect in detail.
  

	\end{rem}


\begin{thm}\label{thm:inverse1}
	Consider the scattering problem \eqref{eq:transea1} associated with two admissible scatterers $\Omega$ and $\widetilde \Omega$. Let $\omega \in \mathbb{R}_+$ be fixed and $\mathbf{d}_\ell$, $\ell=1,2,3$,  be three distinct incident directions from $\mathbb{S}^1$. Let $\mathbf{u}_t^\infty$ and $\widetilde{\mathbf{u}}_t^\infty$ be, respectively, the far-field patterns associated with $\Omega$ and $\widetilde \Omega$.  
	If 
	\begin{equation}\label{eq:farfield}
		\mathbf{u}_t^\infty (\hat{\mathbf{x}}; \mathbf{d}_\ell) = \widetilde{\mathbf{u}}_t^\infty (\hat{\mathbf{x}}; \mathbf{d}_\ell),\quad \hat{\mathbf{x}}\in \mathbb{S}^1, \ell=1,2,3,
	\end{equation} 
then  $\Omega \Delta \widetilde \Omega:=\left(\Omega \backslash \widetilde \Omega  \right) \cup \left(\widetilde \Omega \backslash \Omega \right)$ cannot contain a   convex  corner. Furthermore, if \eqref{eq:farfield} is satisfied, $\Omega$ and $\widetilde \Omega$ are two admissible convex polygons, then one has that 
\begin{equation*}
	\Omega = \widetilde \Omega.
\end{equation*}
\end{thm}

Before presenting the proof of Theorem \ref{thm:inverse1}, we first derive an auxiliary Lemma as follows. The auxiliary Lemma \ref{lem:auxiliary} demonstrates that the elastic total wave fields associated with incident waves given in \eqref{eq:planekpks}, where the incident directions are different, are linearly independent in any setting. The lemma can be proved by following a similar argument to the proof of Theorem 5.1 in \cite{CK}, and we skip the detailed proof.

\begin{lem}\label{lem:auxiliary}
	Let $\mathbf{d}_\ell \in \mathbb{S}^1$, $\ell=1,\ldots,n$, be $n$ vectors which are distinct from each other. Suppose that $\Omega$ is a bounded Lipschitz domain and $\mathbb{R}^2 \backslash \overline{\Omega}$ is connected. Let the incident elastic wave field $\mathbf{u}^i(\mathbf{x};k_p,k_s,\mathbf{d}_\ell)$ be defined in \eqref{eq:planekpks}. Furthermore, assume that the total elastic wave field $\mathbf{u}(\mathbf{x};k_p,k_s,\mathbf{d}_\ell)$ associated with $\mathbf{u}^i(\mathbf{x};k_p,k_s,\mathbf{d}_\ell)$ satisfies \eqref{eq:transea1}. Then the following set of functions is linearly independent:
	\[
	\{\mathbf{u}(\mathbf{x};k_p,k_s,\mathbf{d}_\ell);\mathbf{x}\in D,\ell=1,2,\ldots,n\},
	\]
	where $D\Subset \mathbb{R}^2 \backslash \overline \Omega$ is an open set.
\end{lem}

\begin{proof}[Proof of Theorem \ref{thm:inverse1}]

Let $\mathbf{G}$ denote the unbounded connected component of $\mathbb{R}^2 \backslash (\Omega \cup \widetilde{\Omega})$. By contraction, if \eqref{eq:farfield} holds, $\Omega \Delta \widetilde \Omega$ contains a convex corner. Since partial differential operators $\Delta$ and $\mathcal L_{\lambda,\mu}$ are invariant under rigid motion, without loss of generality, we assume that there exists a corner $S_h \Subset \mathbf{G}\backslash  \Omega$, where $S_h$ is defined in \eqref{eq:Sh} with $h\in \mathbb R_+$. Therefore, the corner point $\mathbf 0$ of $S_h$ satisfies that $\mathbf 0\in \partial \widetilde \Omega$ and $\Gamma_h^\pm \Subset \partial \widetilde \Omega$. 


	Let $\mathbf{u}(\mathbf{x};k_p,k_s,\mathbf{d}_\ell)$ and $\widetilde{\mathbf{u}}(\mathbf{x};k_p,k_s,\mathbf{d}_\ell)$ respectively denote the elastic total wave fields to \eqref{eq:transea1} associate with $\Omega$ and $\widetilde{\Omega}$ with respect to the incident wave $\mathbf{u}^i(\mathbf{x};k_p,k_s,\mathbf{d}_\ell)$. Similarly, assume that  $v(\mathbf{x};k_p,k_s,\mathbf{d}_\ell)$ and $\widetilde{v}(\mathbf{x};k_p,k_s,\mathbf{d}_\ell)$ respectively denote the acoustic total wave fields to \eqref{eq:transea1} associate with $\Omega$ and $\widetilde{\Omega}$ with respect to the incident wave $\mathbf{u}^i(\mathbf{x};k_p,k_s,\mathbf{d}_\ell)$. By virtue of \eqref{eq:transea1} and the Rellich Theorem (cf. \cite{Hahn1998}), we know that
	\begin{equation}\label{eq:totalw}
		\mathbf{u}(\mathbf{x};k_p,k_s,\mathbf{d}_\ell) = \widetilde{\mathbf{u}}(\mathbf{x};k_p,k_s,\mathbf{d}_\ell),\quad \mathbf{x}\in \mathbf{G},\quad \ell=1,2.
	\end{equation}
By virtue of \eqref{eq:totalw}, using transmission conditions in \eqref{eq:transea1},  one has 
\begin{equation}\label{eq:tb}
\widetilde{\mathbf{u}} \cdot \nu =\frac{1}{\rho_b \omega^2} \nabla \widetilde{v} \nu =\mathbf{u} \cdot \nu, \quad T_{\nu} \widetilde{\mathbf{u}} =- \widetilde{v} \nu = T_{\nu} \mathbf{u} \mbox{ on } \Gamma_h^{\pm}.
\end{equation}
Since $S_h \cap \overline \Omega =\emptyset$, it is easy to see that $\mathbf u$ is analytic in $B_h$. Moreover,  $\mathbf u$  and $\widetilde v$ satisfy
\begin{equation}\label{eq:ts}
\mathcal{L}_{\lambda, \mu} \mathbf{u} +\omega^2\rho_e\mathbf{u} =0 \mbox{ in } B_{h}	, \quad \Delta \widetilde{v} +\omega^2 \rho_b\kappa^{-1} \widetilde{v} =0 \mbox{ in } S_{h}.
\end{equation}
By the well-posedness of  \eqref{eq:transea1}, one has $\widetilde{v} \in H^1(S_h)$. According to Lemma \ref{lem:regu1}, we know that $\widetilde{v} \in C^{\alpha}(\overline{ S_h})$. In what follows we consider two cases

\medskip
\noindent{\bf Case 1}. 
Due to the linear dependence of three $\mathbb{C}^3$-vectors, it is easy to see that there exist three complex constants $a_\ell$ such that
\[
\sum\limits_{\ell=1}^3 a_\ell \begin{bmatrix}
	\partial_1  {u}_1(\mathbf{0};k_p,k_s,\mathbf{d}_\ell)\\ \partial_2 u_1 (\mathbf{0};k_p,k_s,\mathbf{d}_\ell)\end{bmatrix} =\mathbf{0},
\]
where $\mathbf{u}(\mathbf{x};k_p,k_s,\mathbf{d}_\ell)=(u_1(\mathbf{x};k_p,k_s,\mathbf{d}_\ell),u_2(\mathbf{x};k_p,k_s,\mathbf{d}_\ell))^\top$. Moreover, there exists at least one $a_\ell$ is not zero. Let
$$
\mathbf{w}^i(\mathbf{x})=\sum\limits_{\ell=1}^3 a_\ell \mathbf{u}^i(\mathbf{x};k_p,k_s,\mathbf{d}_\ell),
$$
which fulfills \eqref{eq:ei1} and is a combination of three elastic incident waves. 
Since the scattering system \eqref{eq:transea1} is linear,  it is easy to see that
\begin{equation}\label{eq:sum1}
	\mathbf{w}(\mathbf{x}) = \sum\limits_{\ell=1}^3 a_\ell \mathbf{u}(\mathbf{x};k_p,k_s,\mathbf{d}_\ell),\quad p(\mathbf{x}) = \sum\limits_{\ell=1}^3 a_\ell
	\widetilde  v(\mathbf{x};k_p,k_s,\mathbf{d}_\ell)
\end{equation}
are total wave fields associated with $\Omega$ and $\widetilde \Omega$ respectively, where $\mathbf{w}^i$ is the corresponding incident wave. 

Next, we distinguish two separate cases. The first case is that $\mathbf{w}(\mathbf{x})\equiv\mathbf{0}$, $\forall \mathbf{x} \in S_h$. In view of \eqref{eq:sum1}, since $a_\ell$ are not all zero and $\mathbf{d}_\ell$ are distinct, we readily have a contradiction by Lemma \ref{lem:auxiliary}. 

For the second case, it yields that $\mathbf{w}(\mathbf{x}) \not \equiv \mathbf{0}$, $\forall \mathbf{x} \in S_h$. In view of \eqref{eq:tb} and \eqref{eq:ts}, we know that $\mathbf{w}(\mathbf{x}) $ and ${p}(\mathbf{x}) $ fulfill
\begin{equation}\label{eq:inverse12 thm}
	\begin{cases}
		\Delta p +\omega^2 \rho_b\kappa^{-1} p =0\quad & \mbox{in}\ S_{h},\medskip\\
		\mathcal{L}_{\lambda, \mu} \mathbf{w} +\omega^2\rho_e\mathbf{w}={\bf 0}\quad & \mbox{in}\ S_{h},\medskip\\
		\mathbf{w}\cdot\nu-\frac{1}{\rho_b\omega^2} \nabla p\cdot\nu=0\quad & \mbox{on}\ \Gamma_{h}^{\pm},\medskip\\
		T_{\nu}\mathbf{w} + p \nu={\bf 0}\quad & \mbox{on}\ \Gamma_{h}^{\pm}.
		\end{cases}
	\end{equation}
	It is evident that $\partial_1 \mathbf w_1(\mathbf 0)$=$\partial_2 \mathbf w_1(\mathbf 0)=0$. Using \eqref{eq:inverse12 thm} and and applying Theorem \ref{thm:thm21}, we have 
\[
\nabla \mathbf {w}(\mathbf{0})=\mathbf 0,
\]
which contradicts to the admissibility condition (b) in Definition \ref{def:admissible}.

\medskip
\noindent{\bf Case 2}. Similarly, since the strain tensor of $\mathbf u$ is symmetric, there exists three nonzero constants such that 
$$
\sum_{\ell=1}^3a_\ell \begin{bmatrix}
	\partial_1  {u}_1(\mathbf{0};k_p,k_s,\mathbf{d}_\ell)-\partial_2{u}_2(\mathbf{0};k_p,k_s,\mathbf{d}_\ell) \\ \partial_2 u_1 (\mathbf{0};k_p,k_s,\mathbf{d}_\ell)+\partial_1 u_2 (\mathbf{0};k_p,k_s,\mathbf{d}_\ell)\end{bmatrix}=\mathbf 0. 
$$
Following a similar argument  for {\bf Case 1}, applying Theorem \ref{thm:thm21}, we can show that
$$
\nabla^s \mathbf w (\mathbf{0})= \partial_1 w_1(\mathbf 0) I,
$$
which contradicts to the admissible condition (a) in Definition \ref{def:admissible}.

The proof is complete.
\end{proof}

\section*{Acknowledgements}

The work of H. Diao is supported by National Natural Science Foundation of China  (No. 12371422) and the Fundamental Research Funds for the Central Universities, JLU (No. 93Z172023Z01). The work of H. Liu and L. Wang is supported by the NSFC/RGC Joint Research Scheme, N\_CityU101/21; ANR/RGC Joint Research Scheme, A\_CityU203/19; and the Hong Kong RGC General Research Funds (projects 11311122, 12301420 and 11300821). The work of Q. Meng is supported by the Hong Kong RGC Postdoctoral Fellowship Scheme (No.: PDFS2324-1S09).

\end{document}